\documentclass[final]{amsart} 
\usepackage{amsmath}
\usepackage{amssymb}
\usepackage{amsthm}
\usepackage{amscd}
\usepackage{mathtools}
\usepackage{marginnote}
\usepackage{bbm}
\usepackage{bm}
\usepackage[notcite,notref]{showkeys}
\usepackage{cite}
\usepackage[colorlinks=true,linkcolor=black,citecolor=black,urlcolor=blue]{hyperref} 

\newtheorem{theorem}{Theorem}[section]
\newtheorem*{theorem*}{Theorem}

\newtheorem*{conjecture*}{Conjecture}

\newtheorem*{question*}{Question}

\newtheorem{lemma}[theorem]{Lemma}
\newtheorem*{lemma*}{Lemma}

\newtheorem{proposition}[theorem]{Proposition}
\newtheorem*{proposition*}{Proposition}

\newtheorem{corollary}[theorem]{Corollary}
\newtheorem*{corollary*}{Corollary}

\theoremstyle{definition}

\newtheorem{definition}[theorem]{Definition}
\newtheorem*{definition*}{Definition}
\newtheorem{remark}[theorem]{Remark}

\newtheorem*{example*}{Example}
\newtheorem*{examples*}{Examples}

\newcommand{\twomat}[4]{\begin{pmatrix} #1 & #2 \\ #3 & #4 \end{pmatrix}}

\renewcommand{\bar}{\overline}
\usepackage{mathtools}

\renewcommand{\AA}{\mathbb{A}}

\newcommand{\FF}{\mathbb{F}}

\newcommand{\PP}{\mathbb{P}}
\newcommand{\QQ}{\mathbb{Q}}

\newcommand{\ZZ}{\mathbb{Z}}

\newcommand{\Ac}{\mathcal{A}}
\newcommand{\Bc}{\mathcal{B}}
\newcommand{\Cc}{\mathcal{C}}

\newcommand{\Fc}{\mathcal{F}}

\newcommand{\Ic}{\mathcal{I}}

\newcommand{\Oc}{\mathcal{O}}

\newcommand{\Sc}{\mathcal{S}}

\newcommand{\Uc}{\mathcal{U}}

\newcommand{\Xc}{\mathcal{X}}
\newcommand{\Yc}{\mathcal{Y}}
\newcommand{\Zc}{\mathcal{Z}}

\newcommand{\Af}{\mathfrak{A}}
\newcommand{\af}{\mathfrak{a}}

\newcommand{\mf}{\mathfrak{m}}

\newcommand{\pf}{\mathfrak{p}}

\newcommand{\rarrow}{\rightarrow}

\newcommand{\onto}{\twoheadrightarrow}
\newcommand{\into}{\hookrightarrow}
\newcommand{\isomto}{\xrightarrow{\sim}}

\newcommand{\tensor}{\otimes}
\newcommand{\normal}{\lhd}

\newcommand{\rhobar}{\bar{\rho}}

\newcommand{\Aut}{\operatorname{Aut}}

\newcommand{\Hom}{\operatorname{Hom}}

\newcommand{\rec}{\operatorname{rec}}
\newcommand{\Gal}{\operatorname{Gal}}

\newcommand{\Ind}{\operatorname{Ind}}
\newcommand{\cInd}{\operatorname{c-Ind}}
\newcommand{\Spec}{\operatorname{Spec}}

\newcommand{\Proj}{\operatorname{Proj}}
\newcommand{\tr}{\operatorname{tr}}

\newcommand{\Frac}{\operatorname{Frac}}

\newcommand{\St}{\operatorname{St}}

\title[Local deformation rings for $GL_2$]{Local deformation rings for $GL_2$ and a Breuil--M\'{e}zard conjecture when $l\neq p$}
\author{Jack Shotton}
\usepackage[utf8]{inputenc}
\newcommand{\gr}{\operatorname{gr}}
\newcommand{\Rbarbox}{\bar{R}\,^\square}
\newcommand{\Rbox}{R^\square}
\newtheorem*{remark*}{Remark}
\begin{document}
\maketitle
\begin{abstract}
  We compute the deformation rings of two dimensional mod $l$ representations of $\Gal (\bar{F}/F)$
  with fixed inertial type, for $l$ an odd prime, $p$ a prime distinct from $l$, and $F/\QQ_p$ a
  finite extension.  We show that in this setting an analogue of the Breuil--M\'{e}zard conjecture
  holds, relating the special fibres of these deformation rings to the mod $l$ reduction of
  certain irreducible representations of $GL_2(\Oc_F)$.
\end{abstract}
\section{Introduction}
\label{sec:introduction}

Let $p$ be a prime, and let $F$ be a finite extension of $\QQ_p$ with absolute Galois group $G_F$.
We study the (framed) deformation rings for two-dimensional mod $l$ representations of $G_F$, where
$l$ is an odd prime distinct from $p$.  More specifically, let $E$ be a finite extension of $\QQ_l$,
with ring of integers $\Oc$, uniformiser $\lambda$, and residue field $\FF$.  Let
\[ \rhobar : G_F \rarrow GL_2(\FF)\] be a continuous representation.  Then there is a universal
lifting (or framed deformation) ring $R^\square(\rhobar)$ parametrising lifts of $\rhobar$.  Our
main result relates congruences between irreducible components of $\Spec R^\square(\rhobar)$ to
congruences between certain representations of $GL_2(\Oc_F)$, where $\Oc_F$ is the ring of integers
of $F$.  Our method is to give explicit equations for the components of $\Spec R^\square(\rhobar)$,
which may be of independent use.

If $\tau : I_F \rarrow GL_2(E)$ is a continuous representation that extends to a representation of
$G_F$ (an \emph{inertial type}), then we say that a representation $\rho : G_F \rarrow
GL_2(\bar{E})$ has type $\tau$ if its restriction to $I_F$ is isomorphic to $\tau$.  Say that an
irreducible component of $\Spec R^\square(\rhobar)$ has type $\tau$ if a Zariski dense subset of its
$\bar{E}$-points correspond to representations of type $\tau$.  We define (definition
\ref{def:cycle-of-type}) a formal sum $\Cc(\rhobar,\tau)$ of irreducible components of the special
fibre $\Spec R^\square(\rhobar)\otimes_\Oc \FF$. For semisimple $\tau$, this is obtained as the
intersection with the special fibre of those components of $\Spec R^\square(\rhobar)$ having type
$\tau$; for non-semisimple $\tau$ this must be slightly modified.

To an inertial type $\tau$ we also associate an irreducible $E$-representation $\sigma(\tau)$ of
$GL_2(\Oc_F)$, by a slight variant on the definition of \cite{2002} (see section \ref{sec:Ktypes}).
For an irreducible $\FF$-representation $\theta$ of $GL_2(\Oc_F)$, define
$m(\theta,\bar{\sigma(\tau)})$ to be the multiplicity of $\theta$ as a Jordan--H\"{o}lder factor of
the mod $\lambda$ reduction of $\sigma(\tau)$.  Then we can state our main theorem (theorem
\ref{thm:BM}):
\begin{theorem*}
  Let $\rhobar : G_F \rarrow GL_2(\FF)$ be a continuous representation.  For each irreducible
  $\FF$-representation $\theta$ of $GL_2(\Oc_F)$, there is a formal sum $\Cc(\rhobar,\theta)$ of
  irreducible components of $\Spec R^\square(\rhobar)\otimes \FF$ such that, for each inertial type
  $\tau$, we have the equality
  \[ \Cc(\rhobar,\tau) = \sum_{\theta} m(\theta, \bar{\sigma(\tau)})\Cc(\rhobar,\theta).\]
\end{theorem*}
In fact the $\Cc(\rhobar,\theta)$ are uniquely determined (at least for those $\theta$ which
actually occur in some $\bar{\sigma(\tau)}$).  

This theorem is an analogue for mod $l$ representations of $G_F$ of the Breuil--M\'{e}zard
conjecture \cite{breuil2002}, which pertains to mod $p$ representations of $G_{\QQ_p}$.  Our
statement is not in the language of Hilbert--Samuel multiplicities used in \cite{breuil2002}, but
rather in the geometric language of \cite{JMJ:9091284}.  The original conjecture of
Breuil--M\'{e}zard was proved in most cases by Kisin \cite{Kisin2009-FontaineMazur}; further cases
were proved by Pa\v{s}k\={u}nas \cite{1209.5205} by local methods, and the full conjecture was proved when
$p>3$ in \cite{1309.1658}.  The conjecture was generalised to $n$-dimensional representations of
$G_F$ in \cite{JMJ:9091284}; the only case known, outside of those just mentioned, is that of
two-dimensional potentially Barsotti--Tate representations (see
\cite{GeeKisin2013-BreuilMezardBarsottiTate}).

In the $l \neq p$ setting, a comparison of special fibres of (very particular) local
deformation rings was used by Taylor in \cite{Taylor2008-AutomorphyII} to prove the change of level
results needed to obtain non-minimal automorphy lifting theorems; this is another motivation for our
result.  

Our method of proof is to completely explicitly determine equations for deformation rings of fixed
type, and indeed obtaining these explicit descriptions is another goal of this paper.  We reduce to
the tamely ramified case, in which we use the relation \[\phi \sigma \phi^{-1} = \sigma^q\] for
$\phi\in G_F$ a lift of Frobenius and $\sigma \in I_F$ a generator of tame inertia.  Since we are
considering lifts $\rho$ of fixed type, and so with fixed characteristic polynomial of
$\rho(\sigma)$, we may use the Cayley--Hamilton theorem to reduce this equation to one of degree at
most two in the entries of $\rho(\phi)$ and $\rho(\sigma)$.  These explicit descriptions show that
the irreducible components of $\Spec R^\square(\rhobar)\otimes \bar{E}$ are always smooth (which is
also proved in \cite{Pilloni2008-2Ddeformationslneqp}), and that the reduced deformation rings in
which the semisimplification of the restriction to inertia is fixed are always Cohen--Macaulay (see
\ref{sec:CM}). It is natural to ask whether these properties persist beyond the case of two
dimensional representations.  We note that the generic fibres of our local deformation rings have
been studied in \cite{Pilloni2008-2Ddeformationslneqp} and
\cite{Reduzzi2013-NumberIrreducibleComponents}, but their methods say little about the integral
structure.

In a forthcoming paper, we will extend theorem \ref{thm:BM} to the case of $n$-dimensional
representations using global methods.

The structure of this paper is as follows.  In section~\ref{sec:preliminaries} we define the
universal deformation rings and show how to reduce their study to the case when $\bar{\rho}$ is
tamely ramified.  We also prove some lemmas that will be useful in the calculations that follow.  In
section~\ref{sec:types} we define the deformation rings with fixed inertial type that we will need,
and discuss the construction of the representations $\sigma(\tau)$.  In
section~\ref{sec:breuil-mezard} we state and prove the main theorem, modulo the calculations of
section~\ref{sec:calculations} and results of section~\ref{sec:repthy}.
Section~\ref{sec:calculations} contains the calculations of explicit equations for local deformation
rings, divided into cases according to the value of $q$ mod $l$.  Finally, in
section~\ref{sec:repthy} we prove the results on the mod $l$ reduction of the
$\sigma(\tau)$ that are stated in section~\ref{sec:reduction-types} (and used in the proof of
theorem~\ref{thm:BM}).

\subsection{Acknowledgements}
\label{sec:acknowledgements}
This work forms part of my Imperial College, London PhD thesis, and I am grateful to my supervisor Toby Gee for
suggesting this topic, for comments on drafts of this paper, and for answering many questions.  I also thank Gebhard
B\"{o}ckle, Kevin Buzzard, David Helm, Yongquan Hu, Tristan Kalloniatis, Lue Pan, and Vytautas Pa\v{s}k\={u}nas for
helpful comments and corrections, and Jack Thorne for encouraging me to investigate the Cohen--Macaulay property of
these deformation rings.

This research was supported by the Engineering and Physical Sciences Research Council and the Philip
Leverhulme Trust, and part of this work was done during a visit to the University of Chicago
supported by the London Mathematical Society and the Cecil King Foundation.

\section{Preliminaries}
\label{sec:preliminaries}
\subsection{Fields and Galois groups.}
\label{sec:fields}
Suppose that $l \neq p$ are primes with $l > 2$.

Let $F/\QQ_p$ be a finite extension with ring of integers $\Oc_F$, maximal ideal $\pf_F$,
uniformiser $\varpi_F$ and residue field $k_F$ of order $q$.  Let $F$ have absolute Galois group
$G_F$, inertia group $I_F$, and wild inertia group $P_F$.  Let $I_F \onto I_F/\tilde{P}_F \cong
\ZZ_l$ be the maximal pro-$l$ quotient of $I_F$, so that $\tilde{P}_F/P_F \cong \prod_{l' \neq l,p}
\ZZ_{l'}$.  Note that $\tilde{P}_F$ is normal in $G_F$ and write $T_F = G_F / \tilde{P}_F$.  The
short exact sequence $1 \rarrow I_F/\tilde{P}_F \rarrow T_F \rarrow G_F/I_F \rarrow 1$ splits, so
that $T_F \cong \ZZ_l \rtimes \hat{\ZZ}$.  We fix topological generators $\sigma$ of this $\ZZ_l$
and $\phi$ of this $\hat{\ZZ}$ such that $\phi$ is a lift of arithmetic Frobenius.  Then the action
of $\hat{\ZZ}$ on $\ZZ_l$ is given by
\begin{equation} \label{fund-relation} \phi \sigma \phi^{-1} = \sigma^q.
\end{equation}
Let $L/F$ be an unramified quadratic extension, with residue field $k_L$.

Now let $E/\QQ_l$ be a finite extension with ring of integers $\Oc$, residue field $\FF$ and
uniformiser $\lambda$.  Let $\epsilon : G_F \rarrow \ZZ_l^\times$ be the $l$-adic cyclotomic
character, and let $\mathbbm{1} : G_F \rarrow \ZZ_l^\times$ be the trivial character.  If $A$ is any
$\Oc$-algebra then we will regard these as maps to $A^\times$ via the structure maps $\ZZ_l \rarrow
\Oc \rarrow A$.

Define two integers $a$ and $b$ by $a = v_l(q-1)$ and $b = v_l(q+1)$, where $v_l$ is the $l$-adic
valuation; at most one of $a$ and $b$ is non-zero, since $l$ is odd.

\subsection{Deformation rings.}

Suppose that $\bar{M}$ is an $n$-dimensional $\FF$-vector space and that $\bar{\rho} : G_F \rarrow
GL(\bar{M})$ is a continuous representation.  Let $(\bar{e}_i)_{i=1}^n$ be a basis for $\bar{M}$, so that
$\bar{\rho}$ gives a map $\bar{\rho} : G_F \rarrow GL_n(\FF)$.
  
Let $\Cc_\Oc$ denote the category of artinian local $\Oc$-algebras with residue field $\FF$, and
$\Cc_\Oc^\wedge$ the category of complete noetherian local $\Oc$-algebras with residue field
$\FF$.  If $A$ is an object of $\Cc_\Oc$ or $\Cc_\Oc^\wedge$, let $\mf_A$ be its maximal ideal.  Define
two functors
\[D(\rhobar), D^\square(\rhobar) : \Cc_\Oc \rarrow \mathbf{Set}\] as follows:
\begin{itemize}
\item $D(\rhobar)(A)$ is the set of equivalence classes of $(M,\iota)$ where: $M$ is a free rank~$n$
  $A$-module, $\rho : G_F \rarrow \Aut_A(M)$ is a continuous homomorphism, and $\iota~:~M\otimes_A
  \FF \isomto \bar{M}$ is an isomorphism commuting with the actions of $G_F$;
  
\item $D^\square(\rhobar)(A)$ is the set of equivalence classes of $(M,\rho, (e_i)_{i=1}^n)$ where: $M$ is
  a free rank $n$ $A$-module, $\rho : G_F \rarrow \Aut_A(M)$ is a continuous homomorphism and
  $(e_i)_{i=1}^n$ is a basis of $M$ as an $A$-module, such that the isomorphism $\iota:M \otimes_A \FF
  \isomto \bar{M}$ defined by $\iota: e_i \otimes 1 \mapsto \bar{e}_i$ commutes with the actions of $G_F$.
\end{itemize}
In the first case, $(M,\rho,\iota)$ and $(M',\rho',\iota')$ are
equivalent if there is an isomorphism $\alpha:M \rarrow M'$, commuting with the actions of $G_F$,
such that $\iota = \iota'\circ \alpha$; in the second case, $(M,\rho,(e_i)_i)$ and
$(M',\rho',(e'_i)_i)$ are isomorphic if the map $M \rarrow M'$ defined by $e_i \mapsto e_i'$ commutes with
the actions of $G_F$.  There is a natural transformation of functors $D^\square(\rhobar)
\rarrow D(\rhobar)$ given by forgetting the basis.

Alternatively, when $\rhobar$ is regarded as a homomorphism to $GL_n(\FF)$, we have the equivalent definitions
\[D^\square(\rhobar)(A) = \{\text{continuous $\rho : G_F \rarrow GL_n(A)$ lifting $\bar{\rho}$}\}\]
and \[D(\rhobar)(A) = \{\text{continuous $\rho : G_F \rarrow GL_n(A)$
  lifting $\bar{\rho}$}\}/\text{conjugacy by $1 + M_n(\mf_A)$}.\] 
  
The functor $D(\rhobar)$ is not usually pro-representable, but the functor $D^\square(\rhobar)$
always is (see, for example, \cite{Kisin2009-ModuliFFGSandModularity} (2.3.4)):
\begin{definition} The \emph{universal lifting ring} (or universal framed deformation ring) of
  $\rhobar$ is the object $R^\square(\rhobar)$ of $\Cc^\wedge_{\Oc}$ that pro-represents the functor
  $D^\square(\rhobar)$.  The universal lift is denoted $\rho^\square : G_F \rarrow
  GL_n(R^\square(\rhobar))$.
\end{definition}

Recall the following calculation (see e.g. \cite{BLGGT2014-PotentialAutomorphy} section 1.2):

\begin{lemma} \label{lem:dimension}
  The ring $R^\square(\rhobar)[1/l]$ is generically formally smooth of dimension $n^2$.
\end{lemma}

The next lemma enables us to reduce to the case where the residual representation is trivial on
$\tilde{P}_F$.  Suppose that $\theta$ is an irreducible $\FF$-representation of $\tilde{P}_F$.  Then
by \cite{ClozelHarrisTaylor2008-Automorphy}, lemma 2.4.11, there is a lift of $\theta$ to an
$\Oc$-representation of $\tilde{P}_F$, which may be extended to an $\Oc$-representation
$\tilde{\theta}$ of $G_\theta$, where $G_\theta$ is the group $\{g\in G_F : g\theta g^{-1} \cong
\theta\}$.  For each irreducible representation $\theta$ of $\tilde{P}_F$, we pick such a
$\tilde{\theta}$ and a finite free $\Oc$-module $N(\theta)$ on which $\tilde{P}_F$ acts as
$\tilde{\theta}$.  If $M$ is a set-finite $\Oc$-module with a continuous action $\rho$ of $G_F$, then
define
\[M_\theta = \Hom_{\tilde{P}_F}(\tilde{\theta}, M).\] The module $M_\theta$ has a natural continuous
action $\rho_\theta$ of $G_\theta$ given by $(gf)(v) = gf(g^{-1}v)$; the subgroup $\tilde{P}_F$ of $G_\theta$
acts trivially.

\begin{lemma} \label{lem:tame-reduction} (Tame reduction) \begin{enumerate}
  \item Let $M$ be a set-finite $\Oc$-module with a continuous action of $G_F$.  Then there is a
    \emph{natural} isomorphism
    \[ M = \bigoplus_{[\theta]}\Ind_{G_\theta}^{G_F}\left( N(\theta)\otimes_{\Oc}
      M_\theta\right),\] where $[\theta]$ runs over $G_F$-conjugacy classes of irreducible
    representations of $\tilde{P}_F$.

  \item The isomorphism of part (1) induces a natural isomorphism of functors:
    \[ D(\rhobar) \isomto \prod_{[\theta]} D(\rhobar_\theta)\]
    where $\theta$ runs through a set of representatives for the $G_F$-conjugacy classes of
    irreducible representations of $\tilde{P}_F$. 

  \item If $R^\square(\rhobar_\theta)$ is the universal framed deformation ring for the
    representation $\rhobar_\theta$ of $G_\theta / \tilde{P}_F$, then
    \[ R^\square(\rhobar) \cong \left(\widehat{\bigotimes}_{[\theta]} R^\square(\rhobar_\theta)
    \right)[[X_1, \ldots, X_{n^2 - \sum n^2_\theta}]]\] where $n_\theta = \dim \rho_\theta$.  This
    isomorphism lies above the isomorphism $D(\rhobar) \isomto \prod_{[\theta]}D(\rhobar_\theta)$ of
    part (2).
\end{enumerate}
\end{lemma}
\begin{proof} The first two parts are in \cite{ClozelHarrisTaylor2008-Automorphy}: part (1) is lemma
  2.4.12 and part (2) is corollary 2.4.13.  Part (3) is the refinement to framed deformations
  obtained by keeping track of a basis in the construction of part (1) of the proposition, as in
  \cite{Choi2009-LocalDeformationRings}, proposition 2.0.5.

  As \cite{Choi2009-LocalDeformationRings} is not easily available, we sketch the argument for part (3): let
  $[\theta_1],[\theta_2],\ldots$ be the $G_F$-conjugacy classes of irreducible
  $\tilde{P}_F$-representations.  Pick left coset representatives $(g_{ij})_j$ for $G_{\theta_i}$ in
  $G_F$.  Write $N_i$ for $N(\theta_i)$, and choose an $\Oc$-basis $(f_{ik})_k$
  of $N_i$.  

  Let $A$ be an object of $\Cc_\Oc$, $M$ be a free rank $n$ $A$-module with a continuous action of
  $G_F$, and $M_{\theta_i}$ be as above.  Given (for each $i$) a basis
  $(e_{il})_{l=1}^{n_{\theta_i}}$ of $M_{\theta_i}$, we can produce a basis $(e_{ijkl})_{j,k,l}$
  of \[M_{\theta_i} = A[G_F] \otimes_{A[G_\theta]} (N_i \otimes_\Oc M_{\theta_i})\] defined by
  \[ e_{ijkl} = g_{ij} \otimes f_{ik} \otimes e_{il}.\] Then $(e_{ijkl})_{i,j,k,l}$ is a basis of
  $M$.  

  Let $\Fc(A)$ be the set of $\mathbf{Y} = (Y_{ijkl,i'j'k'l'})$ which are $n\times n$ matrices of elements of
  $\mf_A$ such that
  \[Y_{ijkl,i'j'k'l'} = 0 \text{ if $i = i'$ and $j = j' = k = k' = 1$}\] (so that $n^2 - \sum
  n_{\theta_i}^2$ `free' entries of $\mathbf{Y}$ remain).  Then $\Fc$ defines a functor on $\Cc_\Oc$
  pro-represented by $\Oc[[X_1,\ldots,X_{n^2 - \sum n_\theta^2}]]$ (the variables $X$ being simply
  an enumeration of those $Y_{ijkl,i'j'k'l'}$ which can be non-zero).

  We then have a natural transformation of functors 
  \[ \Fc \times \prod_{[\theta]}D^\square(\rhobar_\theta) \rarrow D^\square(\rhobar)\] taking the
  tuple $\left(\mathbf{Y}, (M_{\theta_i},\rho_{\theta_i},e_{il})_{i}\right)$ to the tuple
  \[\left(\bigoplus_{i} \Ind_{G_{\theta_i}}^{G_F}(N_i\otimes_\Oc M_{\theta_i}),
  \bigoplus_{i}\Ind_{G_{\theta_i}}^{G_F}(\tilde{\theta}_i\otimes_\Oc \rho_{\theta_i}), (I_n + \mathbf{Y})(e_{ijkl})_{i,j,k,l}\right).\]
  Then one can check (and this is what is done in \cite{Choi2009-LocalDeformationRings}, proposition
  2.0.5) that this is in fact an isomorphism, and so we get the claimed isomorphism of pro-representing objects.
\end{proof}

\subsection{Twisting.}

\begin{lemma} \label{twisting}
Suppose that $\chi : G_F \rarrow \Oc^\times $ is any character.  Then there is a natural isomorphism
\[R^\square(\rhobar) \isomto R^\square(\rhobar \tensor \bar{\chi}).\] 
Moreover, if $\chi_1$ and $\chi_2$ satisfy $\bar{\chi}_1 = \bar{\chi}_2$ then they induce the same
maps $\Rbox(\rhobar)\otimes\FF \isomto \Rbox(\rhobar\otimes \bar{\chi}_i)\otimes \FF$.
\end{lemma}
\begin{proof}
  This follows easily from the isomorphism of functors \[D^\square(\rhobar) \rarrow
  D^\square(\rhobar\tensor\bar{\chi})\] given by tensoring with $\chi$ (remembering that we are
  considering $\Oc$-algebras).  For the last statement, observe that if the functors are restricted
  to $\FF$-algebras then the isomorphism only depends on $\bar{\chi}$.
\end{proof}

Since every $\FF$-valued character lifts to $\Oc$ (using the Teichm\"{u}ller lift) this shows
that $R^\square(\rhobar) \cong R^\square(\rhobar \otimes \bar{\chi})$ for every $\bar{\chi} : G_F
\rarrow \FF^\times$.

We also need the calculation of the universal deformation ring of a character, to which some of our
calculations reduce.  This is completely standard, but we include it as a simple illustration of the
method.

\begin{lemma} \label{character-deformation} Let $\bar{\chi} : G_F \rarrow \FF^\times$ be a
  continuous character.  Then
  \[R^\square(\bar{\chi}) = \frac{\Oc[[X,Y]]}{\left((1 + X)^{l^a} - 1\right)}\] has $l^a$
  irreducible components, indexed by the $l^a$th roots of unity.  They are formally
  smooth of relative dimension one over $\Oc$.
\end{lemma}

\begin{proof}
  By lemma \ref{twisting}, we may take $\bar{\chi}$ to be trivial.  If $\chi$ is any lift of
  $\bar{\chi}$ to an object $A$ of $\Cc_\Oc$, then for $g \in \tilde{P}_F$ we must
  have $\chi(g)^n = 1$ for some $n$ coprime to $l$, and therefore $\chi(g) = 1$, so that we are
  reduced to considering characters of $T_F$.  We must have that $\chi(\sigma)^q = \chi(\sigma)$ and
  $\chi(\sigma) \equiv 1 \mod \mf_A$, and
  therefore that $\chi(\sigma)^{l^a} = 1$.  We are then free to choose $\chi(\phi)$.  Writing
  $\chi(\sigma) = 1+X$ and $\chi(\phi) = 1+Y$, we have shown
  that \[D^\square(\bar{\chi})(A) =
  \Hom_{\Cc_\Oc^\wedge}\left(\frac{\Oc[[X,Y]]}{\left((1 + X)^{l^a} - 1\right)},A
  \right)\] functorially, and so the universal framed deformation ring is as claimed.
\end{proof}

\subsection{Multiplicities and cycles}
\label{sec:mult-cycl}

Suppose that $X$ is a noetherian scheme and that $\Fc$ is a coherent sheaf on $X$.  Let $Y$ be the
scheme-theoretic support of $\Fc$, and let $d \geq \dim Y$.  Let $\Zc^d(X)$ be the free abelian
group on the $d$-dimensional points of $X$; elements of $\Zc^d(X)$ are called $d$-dimensional
cycles.  If $\af \in X$ is a point of dimension $d$ write $[\af]$ for the corresponding element of
$\Zc^d(X)$ and define the multiplicity $e(\Fc,\af)$ to be the length of $\Fc_{\af}$ as an
$\Oc_{Y,\af}$-module (this is zero if $\af \not \in Y$).

\begin{definition} The cycle $Z^d(\Fc)$ associated to $\Fc$ is the element
  \[\sum_{\af}e(\Fc,\af) [\af] \in \Zc^d(X).\]
\end{definition}

If $X = \Spec A$ is affine and $\Fc = \widetilde{M}$ for a finitely generated $A$-module $M$, then we
will write $Z^d(M)$ for $Z^d(\Fc)$.  

If $i:X\rarrow X'$ is a closed immersion of $X$ in a noetherian scheme $X'$, then there is a natural
inclusion $i_*:\Zc^d(X) \rarrow \Zc^d(X')$ for each $d$.  For a coherent sheaf $\Fc$ on $X$ whose
support has dimension at most $d$, we then have \[i_*(Z^d(\Fc)) = Z^d(i_*(\Fc)).\]  We will often use
this compatibility without comment.

A cycle is \textbf{effective} if it is of the form $\sum n_{\mathfrak{a}} [\mathfrak{a}]$ for
$n_\mathfrak{a} \geq 0$.  Say that an effective cycle $C_1$ is a \textbf{subcycle} of an effective
cycle $C_2$ if $C_2-C_1$ is also effective.

\subsection{A determinantal ring.}

For $a$, $b$ and $c$ natural numbers, if $I$ is the ideal generated by the $a \times a$ minors of a $b
\times c$ matrix with independent indeterminant entries over a Cohen--Macaulay ring $A$, then $A/I$
is always Cohen--Macaulay (see \cite{Eisenbud1995-CommutativeAlgebra} theorem 18.18).  We include a simple proof
in the very special case that we need below.

\begin{remark*} The proof given below is incorrect, but the proposition is correct.  See Section~\ref{sec:erratum} for
  details.  We thank Lue Pan for pointing this out.
\end{remark*}

\begin{proposition} \label{prop:minors} Let $k \geq 2$ be an integer and let $A$ be either a field or a
  discrete valuation ring. Let $R = A[X_1,\ldots,X_k,Y_1,\ldots,Y_k]$ and let $I \normal R$ be the
  ideal generated by the $2\times 2$ minors of:
  \[ \begin{pmatrix} X_1 & X_2 & \ldots & X_k \\ Y_1 & Y_2 & \ldots & Y_k \end{pmatrix}.\] Let $S =
  R/I$.  Then $S$ is a Cohen--Macaulay domain and is flat over $A$.  It is Gorenstein if and only if
  $k=2$.
  
  The same is true if we replace $S$ by its completion $S^\wedge$ at the `irrelevant' ideal
  $(X_1,\ldots,X_k,Y_1,\ldots,Y_k)$.
\end{proposition}

\begin{proof}
  Note that $R$ and $S$ are naturally graded $A$-algebras.

  Suppose that $A$ is a field.  It is easy to see that $\Proj(S)$ is a smooth irreducible projective
  variety over $A$ of dimension $k+1$ --- it is covered by the open sets $\{X_i \neq 0\}$ and $\{Y_i
  \neq 0\}$, each of which is isomorphic to $(\AA^1_A\setminus \{0\})\times \AA^k_A$.  Thus $S$ is a
  domain.  We may extend $A$ so that its cardinality is at least $k+1$, and choose pairwise distinct
  $\alpha_1,\ldots, \alpha_k \in A^\times$.

  I claim that $(X_1 - \alpha_1Y_1, \ldots, X_k - \alpha_k Y_k, Y_1 + \ldots + Y_k)$ is a regular
  sequence in $S$.  To see this, observe that $\Proj\left(S/(X_1 - \alpha_1 Y_1, \ldots, X_i - \alpha_i
    Y_i)\right)$ is reduced (we may check this on the affine pieces) and that its irreducible components are
  all of the form
  \[\Proj\left(\frac{R}{(X_j - \alpha_{i_0} Y_j)_{1\leq j\leq k} + (X_j,Y_j)_{1 \leq j \leq i, j \neq
        i_0}}\right)\] for $1 \leq i_0 \leq i$ or of the form
  \[ \Proj(S/(X_1,\ldots, X_i, Y_1,\ldots, Y_i)).\] Now it is easy to check that $X_{i+1} -
  \alpha_{i+1}Y_{i+1}$ (if $i < k$) or $Y_1 +\ldots + Y_k$ (if $i=k$) is a non-zerodivisor on each
  of these components, and so is a non-zerodivisor on $S/(X_1 - \alpha_1 Y_1, \ldots, X_i - \alpha_i
  Y_i)$ as required.

  Now \[S/((X_i - \alpha_i Y_i)_i,Y_1 + \ldots + Y_k)\cong A[Y_2,\ldots,Y_k]/(Y_2,\ldots,Y_k)^2\] is
  Gorenstein if and only if $k=2$, as required.

  If $A$ is a DVR then the following easy lemma (a specialisation of \cite{1111.3654} proposition
  2.2.1) gives the result.

  \begin{lemma} If $A$ is a DVR and $S$ is a finitely generated $A$-algebra such that $S \otimes
    A/\mf_A$ and $S \otimes \Frac A$ are domains of the same dimension, then $S$ is flat over $A$
    (that is, a uniformiser of $A$ is a regular parameter in $S$).
  \end{lemma}
  
  The final statement of the proposition follows from the facts that both localisation and
  completion preserve the properties of being Gorenstein, Cohen--Macaulay, or $A$-flat; $S^\wedge$
  is a domain because its associated graded ring is $S$, which is a domain.
\end{proof}

\section{Types}
\label{sec:types}
\subsection{Inertial types.}
\label{sec:inertial-types}

\begin{definition}
  An \emph{inertial type} $\tau$ (of dimension $n$) is an equivalence class of pairs
  $(r_\tau,N_\tau)$ such that:
  \begin{itemize}
  \item $r_\tau : I_F \rarrow GL_n(\bar{E})$ is a representation with open kernel; 
  \item$N_\tau$ is a nilpotent $n\times n$ matrix over $\bar{E}$; 
  \item$(r_\tau,N_\tau)$ extends to a Weil--Deligne representation of $G_F$.
  \end{itemize}
  In particular, $N_\tau$ commutes with the image of $r_\tau$.  Two such pairs are equivalent if
  they are conjugate by an element of $GL_n(\bar{E})$.
\end{definition}

We say that a continuous representation $\rho : G_F \rarrow GL_n(\bar{E})$ has inertial type
$\tau$ if the restriction to inertia of the associated Weil--Deligne representation is equivalent to
$\tau$.  

We define some particular two-dimensional types which will often arise.  They will all be of the
form $(r,N)$ with $r|_{\tilde{P}_F}$ trivial, and are therefore determined by $r(\sigma)$ and $N$.
Define:
\begin{itemize}
\item $\tau_{\zeta,s}$ by $r(\sigma) = \twomat{\zeta}{0}{0}{\zeta}$ and $N = 0$, where $\zeta$ is an
  $l^a$th root of unity ($s$ is for `split');
\item $\tau_{\zeta,ns}$ by $r(\sigma) = \twomat{\zeta}{0}{0}{\zeta}$ and $N = \twomat{0}{1}{0}{0}$,
  where $\zeta$ is an $l^a$th root of unity ($ns$ is for `non-split');
\item $\tau_{\zeta_1,\zeta_1}$ by $r(\sigma) = \twomat{\zeta_1}{0}{0}{\zeta_2}$ and $N = 0$ where,
  $\zeta_1$ and $\zeta_2$ are distinct $l^a$th roots of unity;
\item $\tau_\xi$ by $r(\sigma) = \twomat{\xi}{0}{0}{\xi^{-1}}$ and $N=0$ where, $\xi$ is a \emph{non-trivial}
  $l^b$th root of unity.
\end{itemize}
To see that $\tau_\xi$ \emph{is} a type, note that if $L/F$ is the unramified quadratic extension,
then there is a character of $G_L/\tilde{P}_F$ mapping $\sigma$ to $\xi$, which when induced to
$G_F$ gives a representation of type $\tau_\xi$.

\subsection{Deformation rings with fixed type.}
\label{sec:deformation-rings-type}
\begin{definition}
  Let $\tau$ be an inertial type.  Then $R^\square(\rhobar, \tau)$ is the maximal reduced,
  $l$-torsion free quotient of $R^\square(\rhobar)$ with the following property: if $x :
  R^\square(\rhobar) \rarrow GL_n(\bar{E})$ is a continuous homomorphism such that the
  associated representation $\rho_x : G_F \rarrow GL_n(\bar{E})$ has type $\tau$, then $x$
  factors through $R^\square(\rhobar,\tau)$.
\end{definition}

The rings $R^\square(\rhobar) \otimes \FF$ and $R^\square(\rhobar,\tau) \otimes \FF$ will occur very
often, and so we denote them respectively by $\Rbarbox(\rhobar)$ and $\Rbarbox(\rhobar,\tau)$.

\emph{From now on suppose that $n=2$.}  Write $\tau = (r_\tau,N_\tau)$ and assume that $E$ is large enough that all of the roots of the
characteristic polynomial of $r_\tau$ lie in $E$.  Let $R^\square(\rhobar,\tau)^\circ$ be the maximal quotient of
$R^\square(\rhobar)$ on which:

\begin{itemize}
\item if $r_\tau$ is not scalar then, for all $g\in I_F$, the characteristic polynomial of $\rho^\square(g)$ agrees with that of
  $r_\tau$;
\item if $r_\tau$ is scalar and $N_\tau=0$ then, for all $g \in I_F$, $\rho^\square(g)$ is scalar and agrees
  with $r_\tau$;
\item if $r_\tau$ is scalar and $N_\tau \neq 0$ then, for all $g \in I_F$, the characteristic
  polynomial of $\rho^\square(g)$ agrees with that of $r_\tau$. Moreover, we have 
  \begin{equation} \label{eq:phi} q(\tr \rho^\square(\phi))^2 = (q+1)^2\det(\rho^\square(\phi)).
  \end{equation}

\end{itemize}

It is clear that these quotients exist and that the conditions imposed are deformation problems for $\bar{\rho}$.

\begin{lemma}
  The ring $R^\square(\rhobar, \tau)$ is a reduced $l$-torsion free quotient of
  $R^\square(\rhobar, \tau)^\circ$.  

  If $N_\tau = 0$, then we have that $R^\square(\rhobar,\tau)$ is equal to the maximal reduced
  $l$-torsion free quotient of $R^\square(\rhobar, \tau)^\circ$.
\end{lemma}
\begin{proof}
  The first part is clear unless $r_\tau$ is scalar and $N_\tau \neq 0$.  In this case, we must show
  that any representation $\rho : G_F \rarrow GL_2(\bar{E})$ of type $\tau$ satisfies equation
  \eqref{eq:phi}.  The Weil--Deligne representation $(r,N)$ corresponding to such a $\rho$ satisfies
  $r |_{I_F} = r_\tau$ and $N\neq 0$.  Then $r(\phi) N = qNr(\phi)$ implies that $r(\phi)$ preserves
  the line $\ker N$ and the quotient $\bar{E}^2/\ker N$.  If it acts as $\alpha$ on the former and
  $\beta$ on the latter then we must have $\alpha = q\beta$; as $\alpha$ and $\beta$ are the
  eigenvalues of $\rho(\phi)$ the equation \eqref{eq:phi} is easily verified.
  
  The final claim follows from the simple observation that any $\bar{E}$-point of
  $R^\square(\rhobar,\tau)^{\circ}$ has associated Galois representation of type $\tau$, except perhaps if $r_\tau$ is
  scalar and $N_\tau \neq 0$.  
\end{proof}

\begin{remark}
  If $R$ is a reduced, $l$-torsion free quotient of $R^\square(\rhobar)$ such that $R^\square(\rhobar,\tau)$ is a
  quotient of $R$, then $R = R^\square(\rhobar, \tau)$ if and only if the closed points of type
  $\tau$ are Zariski dense in $\Spec R[1/l]$.  In our calculations, when this is true it will always
  be clear by inspection.
\end{remark}

\subsection{$K$-Types}
\label{sec:Ktypes}

Let $G = GL_2(F)$, $K = GL_2(\Oc_F)$, and for $N\geq 1$ let $K(N) = 1 + M_2(\pf_F^N)$ and $K_0(N) =
\left\lbrace \twomat{a}{b}{c}{d} \; : \; c \in \pf_F^N\right \rbrace$.  Let $U_0 = \Oc_F^\times$ and
for $N \geq 1$ let $U_N = 1 + \pf_F^N$. The \emph{exponent} of a character $\chi$ of $\Oc_F^\times$
is the smallest $N \geq 0$ such that $\chi$ is trivial on $U_N$.  If $\pi$ is an irreducible
admissible representation of $GL_m(F)$ (we only need $m=1$ and $m=2$) over $\bar{E}$, let
$\rec(\pi)$ be the continuous representation of $W_F$ over $\bar{E}$ associated to $\pi$ under the
local Langlands correspondence (normalised so as to be preserved by automorphisms of $\bar{E}$).

  For each two-dimensional inertial type $\tau = (r_\tau,N_\tau)$, we define an irreducible representation
  $\sigma(\tau)$ by the following recipe:
  \begin{itemize}
  \item If $\tau = \tau_{1,s}$, then $\sigma(\tau)$ is the trivial representation
    of $K$.
  \item If $\tau = \tau_{1,ns}$, then $\sigma(\tau)$ is the inflation to
    $K$ of the Steinberg representation $\St$ of $GL_2(k_F)$.
  \item If $\tau = (\mathbbm{1} \oplus \mathrm{rec}(\epsilon)|_{I_F},0)$ for a non-trivial character
    $\epsilon$ of $F^\times$ of exponent $N$, then \[\sigma(\tau) = \Ind_{K_0(N)}^K \epsilon,\] where
    $\epsilon\left(\twomat{a}{b}{c}{d}\right) = \epsilon(a)$.
  \item If $\tau = (\mathrm{rec}(\pi) |_{I_F},0)$ for a cuspidal representation $\pi$ of $GL_2(F)$,
    then by \cite{BushnellHenniart2006_GL2LocalLanglands}, 15.5 Theorem, there is a certain subgroup $J
    \subset G$, containing the center of $G$ and compact modulo center, and a
    representation $\Lambda$ of $J$ such that \[\pi = \cInd_J^G\Lambda.\]  By conjugating, we may
    suppose that the maximal compact subgroup $J^0$ of $J$ is contained in $K$.  We then have
    \[\sigma(\tau) = \Ind_{J^0}^K (\Lambda|_{J^0}).\]
  \item If $\tau = \tau' \tensor \mathrm{rec}(\chi)|_{I_F}$, then $\sigma(\tau) = \sigma(\tau')
    \tensor (\chi|_{U_0}\circ \det)$.
  \end{itemize}

  This is a slightly modified version of the construction in \cite{2002} --- the construction there
  only depends on $r_\tau$, and agrees with ours whenever $r_\tau$ is not scalar.  The following is an easy consequence of \cite{2002}:

  \begin{proposition}
    If $\sigma(\tau)$ is contained in an irreducible admissible representation $\pi$ of $GL_2(F)$
    and $\rec(\pi) = (r,N)$ then $r|_{I_F} \cong r_\tau$ and either $N\cong N_\tau$ or $N_\tau \neq 0$
    and $N=0$.
    
    If $\pi$ is infinite-dimensional, then the converse is true.
  \end{proposition}

\subsection{Reduction of types}

\label{sec:reduction-types}

Suppose that $\bar{r} : I_F \rarrow GL_2(\bar{\FF})$ is such that $\bar{r}$ extends to $G_F$.  

\begin{definition}
  The set $L(\bar{r})$ is the set of types $\tau$ such that there exists a representation $\rho :
  G_F \rarrow GL_2(\Oc_{\bar{E}})$ of type $\tau$ satisfying
\[\bar{\rho}|_{I_F}\cong \bar{r}.\]
\end{definition}

If $\bar{r}|_{\tilde{P}_F}$ is non-scalar then we abuse notation and also write $L(\bar{r})$ for the
set of $r$ such that $(r,0) \in L(\bar{r})$, as in this case every element of $L(\bar{r})$ is of
this form.

\begin{lemma} \label{lem:type-reduction} Suppose that $\bar{r}$ is trivial on $\tilde{P}_F$.  Then
  each element of $L(\bar{r})$ is one of the types $\tau_{\zeta,s}$, $\tau_{\zeta,ns}$,
  $\tau_{\zeta_1,\zeta_2}$, $\tau_\xi$ defined in section \ref{sec:inertial-types}.
\end{lemma}

\begin{proof} Suppose that $\rho : G_F \rarrow GL_2(\Oc_{\bar{E}})$ is of type $\tau$ and is such
  that $\rhobar|_{I_F} \cong \bar{r}$.  As $\bar{r}|_{\tilde{P}_F}$ is trivial, $\rho$ must also be
  trivial on $\tilde{P}_F$ and its type is determined by the eigenvalues of $\rho(\sigma)$ and by a
  nilpotent matrix $N$ commuting with $\rho(\sigma)$.  Now, the fundamental relation $\phi \sigma
  \phi^{-1} = \sigma^q$ shows that the eigenvalues of $\rho(\sigma)$ are the same (but perhaps in a
  different order) as those of $\rho(\sigma)^q$, and this implies that they are $(q^2-1)$th roots of
  unity.  Moreover, they are congruent to 1 modulo the maximal ideal of $\Oc_{\bar{E}}$, and so must
  in fact be either $l^{a}$th or $l^{b}$th roots of unity (recall that at most one of $a$ and $b$ is
  non-zero, since $l \neq 2$).  If they are distinct $l^a$th roots of unity, then $N$ must be zero
  and $\tau = \tau_{\zeta_1,\zeta_2}$; if they are equal $l^a$th roots of unity then $\tau =
  \tau_{\zeta,s}$ or $\tau_{\zeta,ns}$; if they are $l^b$th roots of unity then they must be $\xi$
  and $\xi^q = \xi^{-1}$ for an $l^b$th root of unity $\xi$.  Moreover the case $\xi = 1$ has
  already been dealt with and so we may assume that $\xi \neq 1$, in which case $N = 0$ and $\tau =
  \tau_\xi$.
\end{proof}

\begin{lemma} \label{lem:type-lifts}
  \begin{enumerate}
  \item Suppose that $\bar{r} |_{\tilde{P}_F}$ is irreducible.  There is a lift $r$ of $\bar{r}$ to
    $GL_2(\bar{E})$, which we fix.  Then $L(\bar{r}) = \{r \otimes \chi\}_\chi$ as $\chi$ runs over
    the set of characters $\chi : I_F \rarrow \bar{E}^\times$ which extend to $G_F$ and reduce to
    the trivial character.
  \item Suppose that $\bar{r} |_{\tilde{P}_F} \cong (\bar{r}_1 \oplus \bar{r}_2)|_{\tilde{P}_F}$
    where $\bar{r}_1$ and $\bar{r}_2$ are distinct characters of $G_F$.  There are lifts $r_1$ and
    $r_2$ of $\bar{r}_1$ and $\bar{r}_2$ to $\bar{E}^\times$, which we fix.  Then $L(\bar{r}) = \{(r_1|_{I_F}
    \otimes \chi_1) \oplus (r_2|_{I_F} \otimes \chi_2)\}_{\chi_1,\chi_2}$ where $\chi_1, \chi_2$ run over all
    pairs of characters $I_F \rarrow \bar{E}^\times$ which extend to $G_F$ and reduce to the trivial
    character.
  \item Suppose that $\bar{r} |_{\tilde{P}_F} \cong (\bar{r}_1 \oplus \bar{r}_1^c)|_{\tilde{P}_F}$
    where $\bar{r}_1$ and $\bar{r}_1^c$ are distinct characters of $G_L$ which are conjugate by an
    element of $G_F$ (recall that $L/F$ is the unramified quadratic extension).  There is a lift
    $r_1$ of $\bar{r}_1$ to $\bar{E}^\times$.  Then $L(\bar{r}) = \{(r_1|_{I_F} \otimes \chi)\oplus
    (r_1^c|_{I_F} \otimes\chi^c)\}_\chi$ as $\chi$ runs over all characters $I_F \rarrow \bar{E}^\times$ which extend to
    $G_L$ and reduce to the trivial character.
  \end{enumerate}
\end{lemma}

\begin{proof}
  This follows from proposition \ref{simplest-case} below; the ingredients in the
  proof of that proposition are lemma \ref{lem:tame-reduction} (reduction to the tame case) and
  lemma \ref{twisting} (lifting ring of a character).
\end{proof}

\begin{lemma} \label{lem:K-type-reduction} If $\tau = (r,0)$ is an inertial type with
  $r|_{\tilde{P}_F}$ non-scalar, then $\bar{\sigma(\tau)}$ is irreducible.  If $\tau'$ is any other
  inertial type, then $\bar{\sigma(\tau')}$ contains $\bar{\sigma(\tau)}$ if and only if $\tau'\in
  L(\bar{r})$ (in which case $\bar{\sigma(\tau)} \cong \bar{\sigma(\tau')}$).
\end{lemma}

\begin{proof}
  These are the results of propositions \ref{prop:type_congruence_implies_K-type_congruence} and
  \ref{prop:K-type_congruence_implies_type_congruence}.
\end{proof}

If $\tau = (r,N)$ with $r|_{\tilde{P}_F}$ scalar, then $\bar{\sigma(\tau)}$ need not be irreducible.
We give the (well-known) analysis of these $\bar{\sigma(\tau)}$ in section
\ref{sec:reduction_tame_types}.  For now, we just give names to the following representations of
$GL_2(k_F)$ (and hence, by inflation, of $K$) over $\FF$:
\begin{itemize}
\item the trivial representation, $\mathbbm{1}$;
\item the Steinberg representation, $\St$ (irreducible if $q \not \equiv -1 \mod l$);
\item if $q \equiv -1 \mod l$, the cuspidal (but not supercuspidal) subrepresentation $\pi_1$ of $\St$.
\end{itemize}

\section{The `Breuil--M\'{e}zard conjecture'}
 \label{sec:breuil-mezard}
 Let $\bar{\rho} : G_F \rarrow GL_2(\FF)$ be a continuous representation, and
 suppose that $E$ is sufficiently large that:
 \begin{itemize}
 \item every subrepresentation of $\bar{\rho} \otimes \bar{\FF}$ is already defined over $\FF$;
 \item $E$ contains all of the $(q^2 - 1)$th roots of unity;
 \item for every $\tau \in L(\bar{\rho}|_{I_F})$, $\sigma(\tau)$ is defined over $E$.
 \end{itemize}
We state our analogue of the Breuil--M\'{e}zard conjecture when $l \neq p$.
 By lemma \ref{lem:dimension} and the fact that $R^\square(\rhobar,\tau)$ is defined to be $\Oc$-flat, we
 have
\[\dim \Rbarbox(\rhobar,\tau) \leq 4.\]
\begin{definition} \label{def:cycle-of-type}
  We associate to each type $\tau = (r,N)$ a cycle $\Cc(\rhobar,\tau)\in
  \Zc^4(\Rbarbox(\rhobar))$ as follows:
  \begin{itemize}
  \item if $N = 0$, set \[\Cc(\rhobar,\tau) = Z^4(\Rbarbox(\rhobar,\tau));\]
  \item if $N \neq 0$ (in which case $r$ must be scalar) let $\tau' = (r,0)$ and
    set \[\Cc(\rhobar,\tau) = Z^4(\Rbarbox(\rhobar,\tau)) +
    Z^4(\Rbarbox(\rhobar,\tau')).\]
  \end{itemize}
\end{definition}

 Then we have
 \begin{theorem}\label{thm:BM}
   For each irreducible $\bar{\FF}$-representation $\theta$ of $GL_2(\Oc_F)$, there is an effective cycle
   $\Cc(\rhobar,\theta) \in \Zc^4(\Rbarbox(\rhobar))$ such that, for any inertial type $\tau$,
   we have an equality of cycles
   \begin{equation}\label{eqn:BM}\Cc(\rhobar,\tau) = \sum_\theta
     m(\theta,\bar{\sigma(\tau)})\Cc(\rhobar,\theta)
   \end{equation}
   where $m(\theta,\bar{\sigma(\tau)})$ is the multiplicity of $\theta$ as a Jordan--H\"{o}lder
   factor of $\bar{\sigma(\tau)}$ and the sum runs over all $\theta$.
 \end{theorem}
 \begin{proof}
   We proceed case by case, using the results of section \ref{sec:reduction-types} and of sections
   \ref{sec:calculations} and \ref{sec:reduction_tame_types} below.

   Suppose that $\rhobar |_{\tilde{P}_F}$ is non-scalar.  Then by lemma \ref{lem:K-type-reduction},
   the representations $\bar{\sigma(\tau)}$ for $\tau \in L(\rhobar|_{I_F})$ are all
   irreducible and isomorphic to a common irreducible representation, which we call $\theta_0$.  By
   corollary \ref{cor:simplest-case}, $\Rbarbox(\rhobar)$ has a unique minimal prime, denoted
   $\af$, which has dimension 4.  So we have
   \[ \Zc^4(\Spec(\Rbarbox(\rhobar))) = \ZZ\cdot [\af].\]
   Define $\Cc(\rhobar,\theta_0) = [\af]$, and $\Cc(\rhobar,\theta) = 0$ for $\theta \neq \theta_0$.  By corollary \ref{cor:simplest-case}, 
   \[ \Cc(\rhobar,\tau) = [\af] = \Cc(\rhobar,\theta_0)\] 
   if $\tau \in L(\rhobar|_{\tilde{P}_F})$ and \[\Cc(\rhobar,\tau) = 0\] otherwise.  In other
   words, for all $\tau$ we have 
   \[\Cc(\rhobar,\tau) = \sum_{\theta} m(\theta,\bar{\sigma(\tau)})\Cc(\rhobar,\theta)\]
   as required.
   
   If $\rhobar|_{\tilde{P}_F}$ is scalar, then we may twist $\rhobar$ by a character of $G_F$ and
   apply lemma \ref{twisting} and so \emph{suppose for the rest of the proof that $\rhobar|_{\tilde{P}_F}$ is trivial}.

   If $q \neq \pm 1 \bmod l$, then $L(\rhobar|_{I_F}) \subset \{\tau_{1,s},\tau_{1,ns}\}$.  By the
   discussion of section \ref{sec:reduction_tame_types}, we have that \[\bar{\sigma(\tau_{1,s})} =
   \mathbbm{1}\] and
   \[\bar{\sigma(\tau_{1,ns})} = \St\] are irreducible and non-isomorphic, and that neither is a Jordan--H\"{o}lder
   factor of any other $\bar{\sigma(\tau)}$.  So the fact that we \emph{can} define the
   $\Cc(\rhobar,\theta)$ so as to satisfy equation \eqref{eqn:BM} is a triviality, as there are no
   relations amongst the $\bar{\sigma(\tau)}$ for different $\tau$.  We work out what the
   $\Cc(\rhobar,\theta)$ are explicitly: for $\theta \neq \mathbbm{1}$ or $\St$ we define
   $\Cc(\rhobar,\theta) = 0$.  Otherwise, there are four cases to consider:
   \begin{itemize}
   \item if $\rhobar(\phi)$ has eigenvalues with ratio not in $\{1, \pm q\}$ then by proposition
     \ref{no-extensions} there is a unique minimal prime $\af_{nr}$ of $\Rbarbox(\rhobar)$.
     In this case, define \begin{align*}\Cc(\rhobar,\mathbbm{1}) &= [\af_{nr}] \\ \Cc(\rhobar,\St) &= [\af_{nr}];\end{align*}
   \item if $\rhobar$ is an extension of the trivial character by itself then by proposition
     \ref{banal} part 1 there is a unique minimal prime $\af_{nr}$ of $\Rbarbox(\rhobar)$.
     In this case, define \begin{align*}\Cc(\rhobar,\mathbbm{1}) &= [\af_{nr}] \\ \Cc(\rhobar,\St) &=
       [\af_{nr}];\end{align*}
  \item if $\rhobar$ is a non-split extension of the trivial character by the cyclotomic character
     then by proposition \ref{banal} part 2 there is a unique minimal prime $\af_{N}$ of
     $\Rbarbox(\rhobar)$. In this case, define \begin{align*}\Cc(\rhobar,\mathbbm{1}) &= 0 \\ \Cc(\rhobar,\St) &= [\af_{N}];\end{align*}
   \item if $\rhobar$ is the direct sum of the trivial character and the cyclotomic character then
     by proposition \ref{banal} part 2 there are two minimal primes of $\Rbarbox(\rhobar)$, denoted
     there by $\af_{nr}$ and $\af_N$.  In this case, define \begin{align*}\Cc(\rhobar,\mathbbm{1})
       &= [\af_{nr}] \\ \Cc(\rhobar,\St) &= [\af_{nr}] + [\af_{N}].\end{align*}
   \end{itemize}
   It is then easy to verify that equation \eqref{eqn:BM} holds; we just do the last case.  We see
   from proposition \ref{banal} part 2 that
   \begin{alignat*}{2}
     \Cc(\rhobar,\tau_{1,s}) &= [\af_{nr}] &= \Cc(\rhobar,\mathbbm{1}) \\
     \Cc(\rhobar,\tau_{1,ns}) &= [\af_{nr}] + [\af_N] &= \Cc(\rhobar,\St)
   \end{alignat*}
   and $\Cc(\rhobar,\tau) = 0$ for all other $\tau$, exactly as required by equation \eqref{eqn:BM}.
   
   If $q = -1 \bmod l$, then $L(\rhobar|_{I_F}) \subset
   \bigcup_{\xi}\{\tau_{1,s},\tau_{1,ns},\tau_{\xi}\}$ for $\xi$ a non-trivial $l^b$th root of
   unity.  By the discussion of section \ref{sec:reduction_tame_types}, we have that
\begin{align*}\bar{\sigma(\tau_{1,s})} &= \mathbbm{1},\\ 
  \bar{\sigma(\tau_{\xi})} &= \pi_1,\\
  \intertext{and}
   \bar{\sigma(\tau_{1,ns})}\,^{ss} &= \mathbbm{1} \oplus \pi_1
 \end{align*} where $\mathbbm{1}$ and $\pi_1$ are irreducible and non-isomorphic, and are not
 Jordan--H\"{o}lder factors of any other $\bar{\sigma(\tau)}$. For $\theta \neq \mathbbm{1}$ or
 $\pi_1$ we define $\Cc(\rhobar,\theta) = 0$.  Otherwise, there are four cases to consider:
   \begin{itemize}
   \item if $\rhobar(\phi)$ has eigenvalues with ratio not in $\{\pm 1\}$ then by proposition
     \ref{no-extensions} there is a unique minimal prime $\af_{nr}$ of $\Rbarbox(\rhobar)$.
     In this case, define \begin{align*}\Cc(\rhobar,\mathbbm{1}) &= [\af_{nr}] \\ \Cc(\rhobar,\pi_1) &= 0;\end{align*}
   \item if $\rhobar$ is an extension of the trivial character by itself then by proposition
     \ref{q+1} part 1 there is a unique minimal prime $\af_{nr}$ of $\Rbarbox(\rhobar)$.
     In this case, define \begin{align*}\Cc(\rhobar,\mathbbm{1}) &= [\af_{nr}] \\ \Cc(\rhobar,\pi_1) &=
       0;\end{align*}
   \item if $\rhobar$ is a non-split extension of the trivial character by the cyclotomic character
     then by proposition \ref{q+1} part 2a there is a unique minimal prime, denoted $\af_{N}$ in
     that proposition, of $\Rbarbox(\rhobar,\tau_{1,ns})$, which we regard as a prime of
     $\Rbarbox(\rhobar)$. In this case, define \begin{align*}\Cc(\rhobar,\mathbbm{1}) &= 0 \\
       \Cc(\rhobar,\pi_1) &= [\af_{N}];\end{align*}
   \item if $\rhobar$ is the direct sum of the trivial character by the cyclotomic character then in
     proposition \ref{q+1} part 2b three four-dimensional primes of $\Rbarbox(\rhobar)$ are
     defined, denoted there $\af_{nr}$, $\af_N$ and $\af_{N'}$.
     In this case, define \begin{align*}\Cc(\rhobar,\mathbbm{1}) &=
       [\af_{nr}] \\ \Cc(\rhobar,\pi_1) &= [\af_{N}] + [\af_{N'}].\end{align*}
   \end{itemize}
   It is then easy to verify that equation \eqref{eqn:BM} holds using proposition
   \ref{no-extensions} in the first case and proposition \ref{q+1} parts 1, 2a, and 2b in the second,
   third, and fourth cases; again we just do the fourth case, which is the most complicated.
   Equation \eqref{eqn:BM} is equivalent to the equations:
   \begin{alignat*}{2}
\Cc(\rhobar,\tau_{1,s}) &= \Cc(\rhobar,\mathbbm{1}) &=& [\af_{nr}]  \\
\Cc(\rhobar,\tau_{1,ns}) &= \Cc(\rhobar,\mathbbm{1}) + \Cc(\rhobar,\pi_1) &=& [\af_{nr}] + [\af_{N}] + [\af_{N'}]  \\
\Cc(\rhobar,\tau_{\xi}) &= \Cc(\rhobar,\pi_1)    &=& [\af_{N}] + [\af_{N'}]\\
\intertext{and}
\Cc(\rhobar,\tau) &= 0 && 
   \end{alignat*}
   if $\tau \not \in\bigcup_\xi\{\tau_{1,s},\tau_{1,ns}, \tau_{\xi}\}$.  But by proposition \ref{q+1} part 2b we have:
\begin{alignat*}{2}
\Cc(\rhobar,\tau_{1,s}) &= Z^4(\bar{R}(\rhobar,\tau_{1,s})) &=& [\af_{nr}]  \\
\Cc(\rhobar,\tau_{1,ns}) &= Z^4(\bar{R}(\rhobar,\tau_{1,s})) + Z^4(\bar{R}(\rhobar,\tau_{1,ns})) &=& [\af_{nr}] + [\af_{N}] + [\af_{N'}]  \\
\Cc(\rhobar,\tau_{\xi}) &= Z^4(\bar{R}(\rhobar,\tau_\xi))   &=& [\af_{N}] + [\af_{N'}]\\
\intertext{and}
\Cc(\rhobar,\tau) &= 0 &&
   \end{alignat*}
   if $\tau \not \in \bigcup_\xi\{\tau_{1,s},\tau_{1,ns}, \tau_{\xi}\}$, as required.

   If $q = 1 \bmod l$, then $L(\rhobar|_{I_F}) \subset
   \bigcup_{\zeta,\zeta_1,\zeta_2}\{\tau_{\zeta,s},\tau_{\zeta,ns},\tau_{\zeta_1,\zeta_2}\}$ for
   $\zeta$, $\zeta_1$ and $\zeta_2$ (possibly trivial) $l^a$th roots of unity with $\zeta_1 \neq
   \zeta_2$.  By the discussion of section \ref{sec:reduction_tame_types}, we have that
\begin{align*}\bar{\sigma(\tau_{\zeta,s})} &= \mathbbm{1},\\ 
  \bar{\sigma(\tau_{\zeta,ns})} &= \St,\\
  \intertext{and}
   \bar{\sigma(\tau_{\zeta_1,\zeta_2})} &= \mathbbm{1} \oplus \St
 \end{align*} where $\mathbbm{1}$ and $\St$ are irreducible and non-isomorphic, and are not
 Jordan--H\"{o}lder factors of any other $\bar{\sigma(\tau)}$. For $\theta \neq \mathbbm{1}$ or
 $\St$ we define $\Cc(\rhobar,\theta) = 0$.  Otherwise, there are four cases to consider:
   \begin{itemize}
   \item if $\rhobar(\phi)$ has eigenvalues with ratio not in $\{\pm 1\}$ then by proposition
     \ref{no-extensions} there is a unique minimal prime $\af_{nr}$ of $\Rbarbox(\rhobar)$.
     In this case, define \begin{align*}\Cc(\rhobar,\mathbbm{1}) &= [\af_{nr}] \\ \Cc(\rhobar,\St) &= [\af_{nr}];\end{align*}
   \item if $\rhobar$ is a ramified extension of the trivial character by itself then by proposition
     \ref{q-1} part 1 there is a unique minimal prime $\af_{N}$ of
     $\Rbarbox(\rhobar,\tau_{1,ns})$ which we regard as a four-dimensional prime of
     $R^\square(\rhobar)$.  In this case, define \begin{align*}\Cc(\rhobar,\mathbbm{1}) &= 0 \\ \Cc(\rhobar,\St)
       &= [\af_N];\end{align*}
   \item if $\rhobar$ is a unramified extension of the trivial character by itself then by
     proposition \ref{q-1} parts 2 and 3 there are four-dimensional primes of
     $\Rbarbox(\rhobar)$ which are denoted there by $[\af_{nr}]$ and $[\af_N]$.
     In this case, define \begin{align*}\Cc(\rhobar,\mathbbm{1}) &= [\af_{nr}] \\
       \Cc(\rhobar,\St) &= [\af_{nr}] + [\af_{N}].\end{align*}
   \end{itemize}
   It is then easy to verify that equation \eqref{eqn:BM} holds using proposition
   \ref{no-extensions} in the first case, proposition \ref{q-1} part 1 in the second case, and
   proposition \ref{q-1} parts 2 and 3 in the third case (according as $\rhobar$ is split or not);
   again we just do the third case, which is the most complicated.  Equation \eqref{eqn:BM} is
   equivalent to the equations:
   \begin{alignat*}{2}
     \Cc(\rhobar,\tau_{\zeta,s}) &= \Cc(\rhobar,\mathbbm{1}) &=& [\af_{nr}]  \\
     \Cc(\rhobar,\tau_{\zeta,ns}) &= \Cc(\rhobar,\St) &=& [\af_{nr}] + [\af_{N}] \\
     \Cc(\rhobar,\tau_{\zeta_1,\zeta_2}) &= \Cc(\rhobar,\mathbbm{1})+ \Cc(\rhobar,\St)    &=& [\af_{nr}] + [\af_{nr}] + [\af_N]\\
     \intertext{and} \Cc(\rhobar,\tau) &= 0 &&
   \end{alignat*}
   if $\tau \not \in
   \bigcup_{\zeta,\zeta_1,\zeta_2}\{\tau_{\zeta,s},\tau_{\zeta,ns},\tau_{\zeta_1,\zeta_2}\}$.  But by
   proposition \ref{q-1} parts 2 and 3 we have:
\begin{alignat*}{2}
\Cc(\rhobar,\tau_{\zeta,s}) &= Z^4(\bar{R}(\rhobar,\tau_{\zeta,s})) &=& [\af_{nr}]  \\
\Cc(\rhobar,\tau_{\zeta,ns}) &= Z^4(\bar{R}(\rhobar,\tau_{\zeta,s})) + Z^4(\bar{R}(\rhobar,\tau_{1,ns})) &=& [\af_{nr}] + [\af_{N}] \\
\Cc(\rhobar,\tau_{\zeta_1,\zeta_2}) &= Z^4(\bar{R}(\rhobar,\tau_{\zeta_1,\zeta_2}))   &=&
2[\af_{nr}] + [\af_N]\\
\intertext{and}
\Cc(\rhobar,\tau) &= 0 &&
   \end{alignat*}
   if $\tau \not \in
   \bigcup_{\zeta,\zeta_1,\zeta_2}\{\tau_{\zeta,s},\tau_{\zeta,ns},\tau_{\zeta_1,\zeta_2}\}$, as required.     
 \end{proof}
 \begin{remark}
   Although the definition of $\Cc(\rhobar,\tau)$ may seem ad-hoc, it in fact has the following
   natural interpretation: it is the reduction modulo $\lambda$ of the cycle in
   $\Zc^4(R^\square(\rhobar))$ obtained by taking the Zariski closure of the closed points $x \in
   \Spec R^\square(\rhobar)[1/l]$ such that $\rec^{-1}(\rho_x)|_K$ contains $\sigma(\tau)$.
 \end{remark}
 \begin{remark}
   We conjecture that the theorem remains true when $l=2$.  
 \end{remark}

\section{Calculations}
\label{sec:calculations}

Let $\rhobar : G_F \rarrow GL_2(\FF)$ be a continuous representation.  The aims of this section are
to give explicit presentations for the rings $R^\square(\rhobar, \tau)$ and to compute the cycles
$Z(\Rbarbox(\rhobar,\tau))\in \Zc^4(\Spec \Rbarbox(\rhobar))$.  We continue to assume that $E$ is
sufficiently large, as defined at the start of the previous section.

\subsection{Simple cases.}

When $\rhobar |_{\tilde{P}_F}$ is not scalar, then lemma
\ref{lem:tame-reduction} allows us to determine the universal framed deformation rings.  Recall
that if $\bar{r} : I_F \rarrow GL_2(\FF)$ is a representation that extends to $G_F$ then we have
defined the set $L(\bar{r})$ of types that lift $\bar{r}$.  

\begin{proposition} \label{simplest-case} If $\rhobar|_{\tilde{P}_F}$ is irreducible, then
  \[R^\square(\rhobar) \cong \Oc[[X,Y,Z_1,Z_2,Z_3]]/((1+X)^{l^a}-1). \] The $l^a$ irreducible
  components of $\Spec R^\square(\rhobar)$ are precisely the $\Spec R^\square(\rhobar,\tau)$ for
  $\tau \in L(\rhobar|_{I_F})$.

  If $\rhobar|_{\tilde{P}_F}$ is a sum of distinct characters which extend to $G_F$, then
  \[R^\square(\rhobar) \cong \Oc[[X_1,X_2,Y_1,Y_2,Z_1,Z_2]]/((1+X_1)^{l^a}-1,(1+X_2)^{l^a}-1).\] The
  $l^{2a}$ irreducible components of $\Spec R^\square(\rhobar)$ are precisely the $\Spec R^\square(\rhobar,\tau)$ for $\tau  \in L(\rhobar|_{I_F})$.

  If $\rhobar|_{\tilde{P}_F}$ is a sum of distinct characters which are conjugate by the non-trivial
  element of $G_L\setminus G_F$, then
  \[R^\square(\rhobar ) \cong \Oc[[X,Y,Z_1,Z_2,Z_3]]/((1+X)^{l^b} - 1).\] The $l^{b}$ irreducible
  components of $\Spec R^\square(\rhobar)$ are precisely the $\Spec R^\square(\rhobar,\tau)$ for $\tau  \in
  L(\rhobar|_{I_F})$.
\end{proposition}
\begin{proof}
  This follows straightforwardly from lemma \ref{lem:tame-reduction}.  Suppose first that $\rhobar
  |_{\tilde{P}_F}$ is irreducible.  Then there is a unique irreducible representation $\theta$ of
  $\tilde{P}_F$ such that $\rhobar_\theta$ (in the notation of lemma \ref{lem:tame-reduction}) is
  non-zero.  For that $\theta$, $\rhobar_\theta$ is an unramified one-dimensional representation of $G_F$.  So
  by lemmas \ref{lem:tame-reduction} and \ref{character-deformation}:
  \[R^\square(\rhobar) \cong R^\square(\rhobar_\theta)[[Z_1,Z_2,Z_3]] \cong \Oc[[X,Y,Z_1,Z_2,Z_3]]/((1+X)^{l^a}-1).\] We have $\rho^\square \cong
  \tilde{\theta} \otimes \chi^\square$ where $\chi^\square$ is the universal character $G_F \rarrow
  R^\square(\rhobar_\theta)^\times$.

  Suppose now that $\rhobar|_{\tilde{P}_F} = \theta_1 \oplus \theta_2$ for distinct characters $\theta_1$ and
  $\theta_2$.  Suppose first that the $\theta_i$ are not $G_F$-conjugate.  As in lemma
  \ref{lem:tame-reduction}, we pick $\Oc$-characters $\tilde{\theta}_1$ and $\tilde{\theta}_2$ of
  $G_F$ lifting and extending $\theta_1$ and $\theta_2$.  Then (in the notation of lemma
  \ref{lem:tame-reduction}) $\rhobar_{\theta_1}$ and $\rhobar_{\theta_2}$ are both unramified
  characters.  By lemmas \ref{lem:tame-reduction} and \ref{character-deformation}:
  \begin{align*}R^\square(\rhobar) &\cong \left(R^\square(\rhobar_{\theta_1}) \hat{\otimes}
      R^\square(\rhobar_{\theta_2})\right)[[Z_1,Z_2]]\\
    & \cong \Oc[[X_1,X_2,Y_1,Y_1,Z_1,Z_2]]/((1+X_1)^{l^a}-1,(1+X_2)^{l^a}-1).
\end{align*}
We have \[\rho^\square \cong \tilde{\theta}_1 \otimes \chi^\square_1 \oplus \tilde{\theta}_2 \otimes
\chi^\square_2\] where each $\chi^\square_i$ is the universal character over
$R^\square(\rhobar_{\theta_i})$.

  Suppose finally that $\theta_1$ and $\theta_2$ are $G_F$-conjugate.  We take $\theta =
  \theta_1$; then $G_\theta = G_L$ where $L$ is a quadratic extension of $F$.  In fact, since
  $\tilde{P}_F \subset G_L$ and $l$ is odd, we must have that $G_L$ is the unramified quadratic
  extension of $F$.  As in lemma
  \ref{lem:tame-reduction}, pick an $\Oc$-character $\tilde{\theta}$ of
  $G_L$ lifting and extending $\theta$.  Then (in the notation of lemma
  \ref{lem:tame-reduction}) $\rhobar_{\theta}$ is an unramified
  character of $G_L$.  By lemmas \ref{lem:tame-reduction} and \ref{character-deformation}:
\begin{align*}R^\square(\rhobar) &\cong R^\square(\rhobar_{\theta})[[Z_1,Z_2,Z_3]]\\
& \cong \Oc[[X,Y,Z_1,Z_2,Z_3]]/((1+X)^{l^b} - 1),
\end{align*}
since $v_l(q^2 - 1) = l^b$.  We have \[\rho^\square \cong \Ind_{G_L}^{G_F}\left(\tilde{\theta}
  \otimes \chi^\square\right) \] where $\chi^\square$ is the universal character over
$R^\square(\rhobar_{\theta})$.

We show that $f:\Spec(R^\square(\rhobar,\tau))\mapsto \tau$ is a bijection from the set of irreducible
components of $\Spec(R^\square(\rhobar))$ to $L(\rhobar|_{I_F})$.  It is easy to see that $f$ is
an injection (from our explicit expressions for $\rho^\square$).  The type of the $\bar{E}$-points
of $\Spec(R^\square(\rhobar,\tau))$ is constant on irreducible components, so to show that a
particular $\tau$ is in the image of $f$ it suffices to produce a lift of $\rhobar$ to $\bar{E}$
of type $\tau$.  Each $\tau \in L(\rhobar|_{I_F})$ is, by definition, the type of a lift of
some $\rhobar'$ with $\rhobar'|_{I_F} \cong \rhobar|_{I_F}$.  But it is clear from the calculations
above that the image of $f$ only depends on $\rhobar|_{I_F}$, and so $f$ is surjective as required.  \end{proof}

\begin{corollary} \label{cor:simplest-case} If $\rhobar |_{\tilde{P}_F}$ is not scalar, then
  $\Rbarbox(\rhobar)$ has a unique minimal prime $\mathfrak{a}$, which has dimension 4.
  For $\tau$ an inertial type we have that \[Z^4(\Rbarbox(\rhobar,\tau)) = [\mathfrak{a}]\]
  if $\tau \in L(\bar{\rho}|_{\tilde{P}_F})$ and $Z^4(\Rbarbox(\rhobar, \tau))=0$ otherwise.
\end{corollary}

We may now assume that $\rhobar |_{\tilde{P}_F}$ is scalar; after a twist (invoking
\cite{ClozelHarrisTaylor2008-Automorphy} lemma 2.4.11 to extend the character occurring in
$\rhobar|_{\tilde{P}_F}$ to the whole Galois group), we may assume that $\rhobar|_{\tilde{P}_F}$ is
\emph{trivial}, so that any lift of $\rhobar |_{\tilde{P}_F}$ is also trivial. In this case, then,
$\rhobar |_{I_F}$ is inflated from a representation of the (procyclic) pro-$l$ group
$I_F/\tilde{P}_F$ over a field of characteristic $l$.  Any irreducible representation in
characteristic $l$ of an $l$-group is trivial, and so $\rhobar|_{I_F}$ must be an extension of the
trivial representation by the trivial representation.  Now, because $\phi \sigma \phi^{-1} =
\sigma^q$, $\rhobar(\phi)$ maps the subspace of fixed vectors of $\rhobar(\sigma)$ to itself;
therefore $\rhobar$ must be an extension of unramified characters.  That is, there is a short exact
sequence
\[0 \rarrow \chi_1 \rarrow \rhobar \rarrow \chi_2 \rarrow 0\]
for unramified characters $\chi_1$ and $\chi_2$.  Such an extension corresponds to an element of
$H^1(G_F, \chi_1\chi_2^{-1})$; by a simple calculation with the local Euler characteristic formula
and local Tate duality, this cohomology group is non-zero if and only if $\chi_1 = \chi_2$ or
$\chi_1 = \chi_2 \epsilon$.  So we can easily deal with the case where neither of these two
possibilities can occur.

\begin{proposition} \label{no-extensions} Suppose that $\rhobar |_{\tilde{P}_F}$ is trivial and that
  $\rhobar(\phi)$ has eigenvalues $\bar{\alpha},\bar{\beta}\in\FF$ with
  $\bar{\alpha}/\bar{\beta} \not\in \{1 ,q, q^{-1}\}$.  Then
  \[R^\square(\rhobar) \cong \frac{\Oc[[A,B,P,Q,X,Y]]}{((1 + P)^{l^a} - 1,(1+Q)^{l^a}-1)},\]
and $\rho^\square(\sigma)$ is diagonalizable with eigenvalues $1+P$ and $1+Q$.

  For $\zeta$ an $l^a$th root of unity (possibly equal to 1), we have
  that \begin{align*}R^\square(\rhobar,\tau_{\zeta,s})&=  \Oc[[A,B,P,Q,X,Y]]/(1+P-\zeta,1+Q-\zeta)\\
&\cong \Oc[[A,B,X,Y]]
\end{align*}
is formally smooth of relative dimension 4 over $\Oc$, and $R^\square(\rhobar,\tau_{\zeta,ns}) = 0$.
If $q = 1 \mod l$ and $\zeta_1,\zeta_2$ are distinct $l^a$th roots of unity, then
  \begin{align*}
    R^\square(\rhobar,\tau_{\zeta_1,\zeta_2}) &= \frac{\Oc[[A,B,P,Q,X,Y]]}{(2 + P + Q - \zeta_1 - \zeta_2,
      PQ - (\zeta_1-1)(\zeta_2 - 1))}\\
& \cong \Oc[[A,B,P,X,Y]]/(1 + P - \zeta_1)(1 + P - \zeta_2).
  \end{align*}

  For all other $\tau$, $R^\square(\rhobar,\tau) = 0$.   

  The ideal $\af_{nr}$ defining $\Rbarbox(\rhobar,\tau_{1,s})$ is the unique minimal prime of
  $\Rbarbox(\rhobar)$.  We have: 
  \[Z^4(\Rbarbox(\rhobar,\tau)) =
  \begin{cases}
    [\af_{nr}] &\text{if $\tau = \tau_{\zeta,s}$}\\
    2[\af_{nr}]& \text{if $\tau = \tau_{\zeta_1,\zeta_2}$}\\
    0 & \text{if $\tau = \tau_{\zeta,ns}$}.
  \end{cases}\]
 \end{proposition}

\begin{proof}
  First note that, by the above cohomology calculation, $\rhobar(\sigma)$ must be trivial.

  Let $\alpha$ and $\beta$ be lifts of $\bar{\alpha}$ and $\bar{\beta}$ to $\Oc$.  Suppose that
  $\mathcal{A}$ is an object of $\Cc_\Oc$ and that $M$ is a free $\mathcal{A}$-module of rank 2 with
  a continuous action of $G_F$ given by $\rho : G_F \rarrow \Aut_{\mathcal{A}}(M)$, reducing to
  $\rhobar$ modulo $\mf_\mathcal{A}$.  Suppose that the characteristic polynomial of $\rho(\phi)$ is
  $(X-\alpha-A)(X-\beta - B)$, where $A,B \in \mf_\mathcal{A}$ -- note that by Hensel's lemma the
  characteristic polynomial \emph{does} have roots in $\mathcal{A}$ reducing to $\bar{\alpha}$ and
  $\bar{\beta}$.  Then there is a decomposition \[M = (\rho(\phi)-\alpha-A)M \oplus
  (\rho(\phi)-\beta-B)M.\] Here it is crucial that $\alpha + A$, $\beta + B$ and $\alpha - \beta + A
  - B$ are all invertible in $\mathcal{A}$.  If $\bar{v}_\alpha,\bar{v}_\beta$ is a basis of
  eigenvectors of $\rhobar(\phi)$ in $M\otimes \FF$ and $v_\alpha, v_\beta$ is a basis of $M$
  lifting $\bar{v}_\alpha,\bar{v}_\beta$ then there are unique $X,Y \in \mf_\mathcal{A}$ such that
  $v_\alpha + X v_\beta, v_\beta + Y v_\alpha$ are eigenvectors of $\rho(\phi)$.  Moreover,
  replacing $(v_\alpha,v_\beta)$ by $(\mu v_\alpha,\mu v_\beta)$ for $\mu \in 1 + \mf_{\Ac}$ does
  not change $X$ and $Y$.

  Therefore we may assume that $\rhobar(\phi) = \twomat{\bar{\alpha}}{0}{0}{\bar{\beta}}$ and that
  \begin{eqnarray*}
    \rho(\phi)  &= \twomat{1}{X}{Y}{1}^{-1}\twomat{\alpha + A}{0}{0}{\beta + B} \twomat{1}{X}{Y}{1} \\
    \rho(\sigma) &= \twomat{1}{X}{Y}{1}^{-1}\twomat{1 + P}{R}{S}{1+Q} \twomat{1}{X}{Y}{1}
  \end{eqnarray*}
  where $X,Y,P,R,S,Q \in\mf_\mathcal{A}$ are uniquely determined by $\rho$.  The equation $\phi \sigma \phi^{-1} = \sigma^q$ implies
  that
  \[ \twomat{\alpha + A}{0}{0}{\beta + B} \twomat{1 + P}{R}{S}{1+Q} \twomat{\alpha + A}{0}{0}{\beta
    + B}^{-1} = \twomat{1 + P}{R}{S}{1+Q}^q.\]

  Looking at the top right and bottom left entries gives that $R = S = 0$.  Then looking at the
  diagonal entries gives that $(1+P)^{q-1} = (1+Q)^{q-1} = 1$, which is equivalent to $(1+P)^{l^a} =
  (1+Q)^{l^a} = 1$.  Thus \[R^\square(\rhobar) \cong \frac{\Oc[[A,B,P,Q,X,Y]]}{((1 + P)^{l^a} -
    1,(1+Q)^{l^a}-1)}.\] The possible inertial types are $\tau_{\zeta,s}$ and
  $\tau_{\zeta_1,\zeta_2}$ ($\tau_{\zeta,ns}$ cannot occur since all lifts are
  diagonalisable). Clearly $R^\square(\rhobar,\tau_{\zeta,s})$ is defined by the equations $1+P =
  1+Q = \zeta$.  The ring $R^\square(\rhobar,\tau_{\zeta_1,\zeta_2})^\circ$ is cut out by the
  equations $2 + P + Q = \zeta_1 + \zeta_2$, $(1+P)(1+Q) = \zeta_1\zeta_2$ and the redundant
  equations $(1+P)^{l^a} = (1+Q)^{l^a} = 1$.  But \[R^\square(\rhobar,\tau_{\zeta_1,\zeta_2})^\circ
  \cong \Oc[[A,B,P,X,Y]]/((1+P-\zeta_1)(1+P-\zeta_2))\] is reduced and $\lambda$-torsion
  free and so is equal to $R^\square(\rhobar,\tau_{\zeta_1,\zeta_2}).$
  
  For the reduction modulo $\lambda$, simply note that: 
  \begin{align*}
    \Rbarbox(\rhobar) &= \FF[[A,B,P,Q,X,Y]]/(P^{l^a},Q^{l^a}) \\
    \Rbarbox(\rhobar,\tau_{\zeta,s}) &= \FF[[A,B,P,Q,X,Y]]/(P,Q) \\
    \intertext{and}
  \Rbarbox(\rhobar,\tau_{\zeta_1,\zeta_2}) &= \FF[[A,B,P,Q,X,Y]]/(P^2,Q^2,P+Q).
  \end{align*}
  So $\af_{nr} = (P,Q)$ is the unique minimal prime of $\Rbarbox(\rhobar)$ and the
  multiplicities are as claimed.
\end{proof}

We extract one part of the proof of this proposition for future use:

\begin{lemma} \label{diagonalization} If $\rhobar(\phi)$ has distinct eigenvalues, we may assume
  that it is diagonal. In that case, there exists a unique matrix $\twomat{1}{X}{Y}{1} \in
  GL_2(R^\square(\rhobar))$, reducing to the identity modulo the maximal ideal, such that
  $\rho^\square(\phi) = \twomat{1}{X}{Y}{1}^{-1} \Phi \twomat{1}{X}{Y}{1}$ for a \emph{diagonal}
  matrix $\Phi$.
\end{lemma}
\begin{proof}
  This is simply the first half of the proof of the previous proposition.
\end{proof}

\subsection{$q\neq \pm 1 \bmod l$}

Suppose that $q \neq \pm 1 \bmod l$.  By lemma \ref{no-extensions}, we have already dealt with the
cases in which the eigenvalues of $\rhobar(\phi)$ are not in the ratio $1$ or $q^{\pm 1}$.  All
other cases are dealt with by the following (after twisting and conjugating $\rhobar$).  Note that,
by lemma~\ref{lem:type-reduction}, the only possible types when $\rhobar |_{\tilde{P}_F}$ is trivial
are $\tau_{1,s}$ and $\tau_{1,ns}$.

\begin{proposition}\label{banal} Suppose that $q \neq \pm 1 \mod l$, and that $\rhobar |_{\tilde{P}_F}$
  is trivial.  Then
  \begin{enumerate}

  \item Suppose that $\rhobar(\sigma)$ is trivial, and that $\rhobar(\phi) =
    \twomat{1}{\bar{y}}{0}{1}$ for $\bar{y} \in \FF$.  Then $R^\square(\rhobar, \tau_{1,s}) =
    R^\square(\rhobar)$ is formally smooth of relative dimension 4 over $\Oc$, while
    $R^\square(\rhobar,\tau_{1,ns}) = 0$.

  \item Suppose that $\rhobar(\sigma) = \twomat{1}{\bar{x}}{0}{1}$ and $\rhobar(\phi) =
    \twomat{q}{0}{0}{1}$.

    If $\bar{x}\neq 0$, then $R^\square(\rhobar,\tau_{1,ns}) = R^\square(\rhobar)$ is formally smooth
    of relative dimension 4 over $\Oc$, while $R^\square(\rhobar,\tau_{1,s}) = 0$.

    If $\bar{x} = 0$ then \[R^\square(\rhobar) \cong \Oc[[X_1,\ldots,X_5]]/(X_1X_2).\] The quotients
    by the two minimal primes are $R^\square(\rhobar,\tau_{1,s})$ and $R^\square(\rhobar,
    \tau_{1,ns})$, so that both are formally smooth of relative dimension 4 over $\Oc$.  The minimal
    primes $\af_{nr}$ and $\af_N$ of $\Rbarbox(\rhobar)$ which respectively define
    $\Rbarbox(\rhobar,\tau_{1,s})$ and $\Rbarbox(\rhobar,\tau_{1,ns})$ are
    distinct.
  \end{enumerate}
\end{proposition}
\begin{proof}
 For the first part, write 
\begin{align*}
  \rho^\square(\sigma) &= \twomat{1+A}{B}{C}{1+D} \\
  \rho^\square(\phi) &= \twomat{1+P}{y + R}{S}{1+Q}
\end{align*} 
where $y$ is a lift of $\bar{y}$ (taken to be zero if $\bar{y} = 0$) and $A,B,C,D,P,Q,R,S \in \mf$.

Let $I = (A,B,C,D)$.  Considering the equation $\rho^\square(\phi)\rho^\square(\sigma) =
  \rho^\square(\sigma)^q\rho^\square(\phi)$ modulo the ideal $I\mf$ gives equations
  $Cy\equiv(q-1)A$, $B+Dy \equiv qAy+qB$, $C \equiv qC$ and $(q-1)D + qCy \equiv 0$, all modulo
  $I\mf$.  As $q \neq 1 \mod l$ we find that $I = I\mf $.  Therefore, by Nakayama's lemma, $I = 0$
  and $\rho^\square$ is unramified.  So $R^\square(\rhobar) = R^\square(\rhobar, \tau_{1,s}) \cong
  \Oc[[P,Q,R,S]]$ as claimed.  Note that this proof is still valid if $q = -1 \mod l$.
  
The proof of the second part is similar.  By lemma~\ref{diagonalization}, we may write 
\begin{align*}
  \rho^\square(\sigma) &= \twomat{1}{X}{Y}{1}^{-1} \twomat{1+A}{x+B}{C}{1+D} \twomat{1}{X}{Y}{1}\\
  \rho^\square(\phi)   &= \twomat{1}{X}{Y}{1}^{-1} \twomat{q(1+P)}{0}{0}{1+Q} \twomat{1}{X}{Y}{1}
\end{align*}
with $x$ a lift of $\bar{x}$ (taken to be zero if $\bar{x} = 0$) and $A,B,C,D,X,Y,P,Q \in \mf$.

Let $I = (A,C,D)$.  Considering the relation $\phi \sigma \phi^{-1} = \sigma^q$ modulo $I\mf$ and
applying Nakayama's lemma as before now yields $A = C = D = 0$ (using that $q^2 \neq 1 \mod l$).
The relation (not modulo any ideal) gives that $(x+B)(P-Q) = 0$, and it is easy to see if this
equality holds then the given formulae for $\rho^\square$ do indeed define a representation so
that \[R^\square(\rhobar) = \frac{\Oc[[B,P,Q,X,Y]]}{((x+B)(P-Q))}.\]

If $\bar{x} \neq 0$ then this implies that $P = Q$.  Then $R^\square(\rhobar) = \Oc[[B,P,X,Y]]$.  It
is clear that $R^\square(\rhobar) = R^\square(\rhobar,\tau_{1,ns})$, and the proposition follows.

If $\bar{x} = 0$ then, writing $U = P-Q$, we have $R^\square(\rhobar) = \Oc[[B,P,U,X,Y]]/(BU)$. In
these coordinates, it is clear from the description of $\rho^\square$ that
\[R^\square(\rhobar,\tau_{1,s}) = R^\square(\rhobar)/(B)\] and \[R^\square(\rhobar, \tau_{1,ns}) =
R^\square(\rhobar)/(U).\]  The proposition follows.
\end{proof}

\subsection{$q=-1\bmod l$}
Suppose that $q = -1 \bmod l$.  By proposition \ref{no-extensions}, we have already dealt with the
cases in which the eigenvalues of $\rhobar(\phi)$ are not in the ratio $1$ or $-1$.  All other cases are dealt with by
the following (after twisting and conjugating $\rhobar$).  By lemma \ref{lem:type-reduction},
the only possible types when $\rhobar |_{\tilde{P}_F}$ is trivial are
$\tau_{1,s}$, $\tau_{1,ns}$ and $\tau_{\xi}$ for $\xi$ a non-trivial $l^b$th root of
unity.
\begin{proposition} \label{q+1}
  Suppose that $q = -1 \bmod l$ and that $\rhobar |_{\tilde{P}_F}$ is trivial.
  \begin{enumerate}
  \item Suppose that $\rhobar(\sigma) = \twomat{1}{0}{0}{1}$ and $\rhobar(\phi) =
    \twomat{1}{\bar{y}}{0}{1}$ for $\bar{y} \in \FF$. Then \[R^\square(\rhobar,\tau_{1,s}) =
    R^\square(\rhobar)\] is formally smooth of relative dimension 4 over $\Oc$, while
    \[R^\square(\rhobar,\tau_{1,ns}) = R^\square(\rhobar,\tau_{\xi}) = 0.\]  If $\af_{nr}$ is the unique
    minimal prime of $\Rbarbox(\rhobar)$, then we have 
   \[ Z^4(\Rbarbox(\rhobar,\tau)) = \begin{cases}
     [\af_{nr}] & \text{if $\tau = \tau_{1,s}$} \\
     0 & \text{if $\tau = \tau_{1,ns}$} \\
     0 & \text{if $\tau = \tau_{\xi}$.}
     \end{cases}\]
  \item Suppose that $\rhobar(\sigma) = \twomat{1}{\bar{x}}{0}{1}$ and $\rhobar(\phi) =
    \twomat{q}{0}{0}{1}$ for $\bar{x} \in \FF$.  
    \begin{enumerate}
    \item If $\bar{x} \neq 0$, then $R^\square(\rhobar,\tau_{1,ns})$ and
      $R^\square(\rhobar,\tau_{\xi})$ are formally smooth of relative dimension 4 over $\Oc$, while
      $R^\square(\rhobar,\tau_{1,s}) = 0$.  If $\af_{N}$ is the prime ideal of
      $\Rbarbox(\rhobar)$ cutting out $\Rbarbox(\rhobar, \tau_{1,ns})$ then we have
      \begin{equation}
      \label{q+1-cases-2} Z^4(\Rbarbox(\rhobar,\tau)) = \begin{cases}
     0 & \text{if $\tau = \tau_{1,s}$} \\
     [\af_N] & \text{if $\tau = \tau_{1,ns}$} \\
     [\af_N] & \text{if $\tau = \tau_{\xi}$.}
     \end{cases}
   \end{equation}

 \item If $\bar{x} = 0$, then $R^\square(\rhobar,\tau_{1,s})$ is formally smooth of relative
   dimension 4 over $\Oc$ and \[R^\square(\rhobar,\tau_{1,ns})\cong
   \frac{\Oc[[X_1,\ldots,X_6]]}{\left((X_1,X_3) \cap (X_2,X_3-(q+1))\right)}\] is a
   non-Cohen--Macaulay ring of relative dimension 4 over $\Oc$.  Its spectrum is the scheme
   theoretic union of two formally smooth components that do not intersect in the generic fibre.
   Lastly, \[R^\square(\rhobar,\tau_{\xi}) \cong
   \frac{\Oc[[X_1,\ldots,X_5]]}{(X_1X_2-(\xi-\xi^{-1})^2)}\] is a complete intersection domain of
   relative dimension 4 over $\Oc$ with formally smooth generic fibre.  If $\af_{nr}$ is the prime
   of $\Rbarbox(\rhobar)$ corresponding to $\Rbarbox(\rhobar,\tau_{1,s})$ and $\af_N, \af_N'$ are
   the prime ideals of $\Rbarbox(\rhobar)$ corresponding to the two minimal primes of
   $\Rbarbox(\rhobar,\tau_{1,ns})$, then we have
    \begin{equation}
    \label{q+1-cases-3}Z^4(\Rbarbox(\rhobar,\tau)) = \begin{cases}
     [\af_{nr}] & \text{if $\tau = \tau_{1,s}$} \\
     [\af_{N}] + [\af_{N'}] & \text{if $\tau = \tau_{1,ns}$} \\
     [\af_N] + [\af_{N'}] & \text{if $\tau = \tau_{\xi}$.}
     \end{cases}
   \end{equation}
    \end{enumerate}
  \end{enumerate}
\end{proposition}
\begin{proof}
  The proof of the first part is identical to that of proposition \ref{banal}, part 1. 

For the second part, by lemma \ref{diagonalization} we may write 
\begin{align*}
\rho^\square(\sigma) &= \twomat{1}{X}{Y}{1}\twomat{1+A}{x+B}{C}{1+D}\twomat{1}{X}{Y}{1} \\
\rho^\square(\phi)   &= \twomat{1}{X}{Y}{1}\twomat{-(1+P)}{0}{0}{1+Q}\twomat{1}{X}{Y}{1}
\end{align*}
with $x$ a lift of $\bar{x}$ (taken to be zero if $\bar{x} = 0$) and $A,B,C,D,X,Y,P,Q \in \mf$.

Firstly, it is clear that $R^\square(\rhobar,\tau_{1,s}) = 0$ if $\bar{x} \neq 0$
and \[R^\square(\rhobar,\tau_{1,s}) \cong \Oc[[P,Q,X,Y]]\] if $\bar{x} = 0$.

Next we deal with $\tau_{1,ns}$.  On $R^\square(\rhobar, \tau_{1,ns})$ we have the equations 
\begin{align*}\tr(\rho^\square(\sigma)) &= 2\\
  \det (\rho^\square(\sigma)) &= 1\\
  q \tr(\rho(\phi))^2 &= (q+1)^2 \det(\rho(\phi)) \\
  \intertext{and} 
\rho^\square(\phi)\rho^\square(\sigma)\rho^\square(\phi)^{-1} &=
  \rho^\square(\sigma)^{q}.
\end{align*} The first two of these may be rewritten as \[A = -D\] and \[A^2 + (x+B)(C) = 0\] and
the third can be written as \[(q + 1 + P + qQ)(q + 1 + Q + qP)=0.\]

By the Cayley--Hamilton theorem, $(\rho^\square(\sigma)-1)^2 = 0$ on
$R^\square(\rhobar,\tau_{1,ns})^\circ$; it follows that $\rho^\square(\sigma)^q - 1 =
q(\rho^\square(\sigma)-1)$ on $R^\square(\rhobar, \tau_{1,ns})^\circ$ and so the relation $\phi
\sigma\phi^{-1} = \sigma^q$ together with $D = -A$ yields the equation:
\[\twomat{A}{-(x+B)\frac{1+P}{1+Q}}{-C\frac{1+Q}{1+P}}{-A} = \twomat{qA}{q(x+B)}{qC}{-qA}.\]
Equating coefficients and using that 2 and $q-1$ are invertible we obtain that $A=D= 0$ and that 
\begin{align}
  (x+B)(q + 1 + qQ + P) &= 0 \\
  C(q + 1 + Q + qP) &= 0 \\
  (x+B)C &= 0\\
  (q + 1 + Q + qP)(q + 1 + qQ + P) &= 0
\end{align}
is a complete set of equations cutting out $R^\square(\rhobar,\tau_{1,ns})^\circ$ (the last two
equations being, respectively, the conditions on $\det(\rho^{\square}(\sigma))$ and on
$\rho^\square(\phi)$).

If $\bar{x} \neq 0$ then these equations are equivalent to $q+ 1 + qQ + P = 0$ and $C = 0$ and so we
see that \[R^\square(\rhobar,\tau_{1,ns}) \cong \Oc[[B,P,X,Y]].   \]

If $\bar{x} = 0$ then the left hand sides of the four equations given generate the ideal \[I=(B,q +
1 + Q + qP) \cap (C,q + 1 + qQ + P)\] in $\Oc[[B,C,P,Q,X,Y]]$.  Since $\Oc[[B,C,P,Q,X,Y]]/I$ is
reduced and $\lambda$-torsion free and a Zariski dense set of its $\bar{E}$-points have type
$\tau_{1,ns}$, it is equal to $R^\square(\rhobar,\tau_{1,ns})$.  After the change of variables $X_3
= \frac{q(q+1+Q + qP)}{(q-1)(1+P)}$, $(X_1,X_2,X_4,X_5,X_6) = (B,C,P,X,Y)$ we get the presentation
given in the proposition.

Let \[\Sc = \frac{\Oc[[X_1,X_2,X_3]]}{(X_1,X_3) \cap (X_2,X_3 - (q+1))}.\] Then $\Sc$ has dimension
two.  We show that $\Sc$ is not Cohen--Macaulay; the same is then true for
$R^\square(\rhobar,\tau_{1,ns})$. Now, $\lambda$ is a non-zerodivisor in $\Sc$, and \[\Sc/\lambda =
\frac{\FF[[X_1,X_2,X_3]]}{(X_1X_2,X_1X_3,X_2X_3,X_3^2)}.\] The maximal ideal of $\Sc/\lambda$ is
annihilated by $X_3$, and $X_3 \neq 0$ in $\Sc/\lambda$.  So $\Sc/\lambda$, and hence $\Sc$, is not
Cohen--Macaulay.  The remaining statements about $R^\square(\rhobar,\tau_{1,ns})$ are clear.

Now suppose that $\tau = \tau_\xi$.  On $R^\square(\rhobar,\tau_{\xi})$ we have
\begin{align*}\tr(\rho^\square(\sigma)) &= \xi + \xi^{-1} \\ 
\det (\rho^\square(\sigma)) & =1\\ 
\intertext{and}
\rho^\square(\phi)\rho^\square(\sigma)\rho^\square(\phi)^{-1} &= \rho^\square(\sigma)^q.
\end{align*}
The first two of these may be rewritten as \[A + D = \xi + \xi^{-1} - 2\] and \[AD - (x+B)C = 2 -
\xi - \xi^{-1}.\] By the Cayley--Hamilton theorem, $(\rho^\square(\sigma) -
\xi)(\rho^\square(\sigma) - \xi^{-1}) = 0$.  As \[T^q \equiv \xi + \xi^{-1} - T \mod (T -
\xi)(T-\xi^{-1}) \] in $\ZZ[T]$, the relation $\phi\sigma\phi^{-1} = \sigma^q$ yields
\[\twomat{1+A}{-(x+B)\frac{1+P}{1+Q}}{-C\frac{1+Q}{1+P}}{1+D} = \twomat{\xi + \xi^{-1} - 1 -
  A}{-(x+B)}{-C}{\xi + \xi^{-1} - 1 - D}.\]
Equating coefficients and combining with the equation $\det(\rho^\square(\sigma)) = 1$ we get:
\begin{align}
  A = D &= \frac{\xi + \xi^{-1}}{2} - 1\\ 
(x+B)(P-Q) &= 0 \\ 
C(P-Q) &= 0 \\ 
4(x+B)C &= (\xi - \xi^{-1})^2.
\end{align}
If $\bar{x} \neq 0$ then these equations are equivalent to $P=Q$ and $C = \frac{(\xi - \xi^{-1})^2}{4(x+B)}$, so that \[R^\square(\rhobar, \tau_\xi) \cong \Oc[[X,Y,B,P]].\]

If $\bar{x} = 0$, then the equations imply that 
\[ 0 = BC(P-Q) = \left(\frac{\xi - \xi^{-1}}{2}\right)^2(P-Q)\]
and hence that $P=Q$, as $R^\square(\rhobar,\tau_\xi)$ is $\lambda$-torsion free by definition.
Thus 
\[R^\square(\rhobar,\tau_\xi) \cong \frac{\Oc[[X,Y,B,C,P]]}{(4BC - (\xi - \xi^{-1})^2)}.\]
The remaining statements about $R^\square(\rhobar,\tau_\xi)$ are clear.

Now we calculate the various $Z^4(\Rbarbox(\rhobar,\tau))$.  For part 1, this is
trivial.  For part 2, we have computed each $\Rbarbox(\rhobar,\tau)$ as a quotient of the ring
$\FF[[A,B,C,D,P,Q,X,Y]]$ by an ideal which we call $I(\tau)$.  We see that if $\bar{x} \neq 0$ then
$I(\tau_{1,ns}) = I(\tau_{\xi})$,
and $\Rbarbox(\rhobar,\tau_{1,s}) = 0$, from which equation \ref{q+1-cases-2} follows.  If $\bar{x} = 0$
then \begin{align*}
I(\tau_{1,s}) &= (A,B,C,D)\\ 
I(\tau_{1,ns}) &= (A,D,BC,B(Q-P),C(Q-P),(Q-P)^2)\\ \intertext{and}
I(\tau_{\xi}) &= (A,D,BC,Q-P).
\end{align*}

The minimal primes above these $I(\tau)$ in $\FF[[A,\ldots, Y]]$ are $\af_{nr} = (A,B,C,D)$, $\af_N
= (A,C,D,Q-P)$ and $\af_{N'} = (A,B,D,Q-P)$; the multiplicities in equation \ref{q+1-cases-3} are
then easily verified.
\end{proof}

\begin{remark} \label{rem:CM-q+1} When $\rhobar$ is unramified and $\rhobar(\phi) =
  \twomat{q}{0}{0}{1}$, the ring $R^\square(\rhobar, \tau_{1,ns})$ is not Cohen--Macaulay.
  However the ring $R^\square(\rhobar,\text{unip})$, defined to be the maximal reduced quotient of
  $R^\square(\rhobar)$ on which $\rho^\square(\sigma)$ is unipotent (so that $\Spec
  R^\square(\rhobar,\text{unip})$ is the scheme-theoretic union of $\Spec R^\square(\rhobar,\tau_{1,s})$
  and $\Spec R^\square(\rhobar, \tau_{1,ns})$ in $\Spec R^\square(\rhobar)$), \emph{is}
  Cohen--Macaulay.  Indeed it is easy to see from the above proof that 
  \[R^\square(\rhobar,\text{unip}) \cong
  \frac{\Oc[[X_1,\ldots,X_6]]}{(X_1X_2,X_1(X_3-(q+1)),X_2X_3)}\] which is Cohen--Macaulay
  ($(\lambda,X_1+X_2+X_3,X_4,X_5,X_6)$ is a regular sequence).
\end{remark}

\subsection{$q = 1 \bmod l$}

Suppose that $q = 1 \bmod l$.  By proposition \ref{no-extensions}, we have already dealt with the
cases in which the eigenvalues of $\rhobar(\phi)$ are distinct.  All other cases are dealt with by
the following (after twisting and conjugating $\rhobar$).  Note that by lemma \ref{lem:type-reduction},
the only possible types when $\rhobar |_{\tilde{P}_F}$ is trivial are
$\tau_{\zeta,s}$, $\tau_{\zeta,ns}$ and $\tau_{\zeta_1,\zeta_2}$ for $\zeta$ any $l^a$th root of
unity and $\zeta_1,\zeta_2$ any distinct $l^a$th roots of unity.
\begin{proposition} \label{q-1}
  Suppose that $q = 1 \bmod l$ and that $\rhobar |_{\tilde{P}_F}$ is trivial.  Suppose that
  $\rhobar(\sigma) = \twomat{1}{\bar{x}}{0}{1}$ and $\rhobar(\phi) = \twomat{1}{\bar{y}}{0}{1}$ for
  $\bar{x},\,\bar{y} \in \FF$.
  \begin{enumerate}
  \item If $\bar{x} \neq 0$ then $R^\square(\rhobar,\tau_{\zeta,s}) = 0$, while
    $R^\square(\rhobar,\tau_{\zeta,ns})$ and $R^\square(\rhobar, \tau_{\zeta_1,\zeta_2})$ are
    formally smooth over $\Oc$ of relative dimension 4.  

    If $\af_N$ is the four-dimensional prime of $\Rbarbox(\rhobar)$ corresponding to
    $\Rbarbox(\rhobar,\tau_{1,ns})$ then we have:     
    \begin{equation}
      \label{eq:q-1case1}
      Z^4(\Rbarbox(\rhobar,\tau)) =
      \begin{cases}
        0 & \text{if $\tau = \tau_{\zeta,s}$} \\
        [\af_N] & \text{if $\tau = \tau_{\zeta,ns}$}\\
        [\af_N] & \text{if $\tau = \tau_{\zeta_1,\zeta_2}$.}
      \end{cases}
    \end{equation}

  \item If $\bar{x} = 0$ and $\bar{y} \neq 0$, then $R^\square(\rhobar,\tau_{\zeta,s})$ and
    $R^\square(\rhobar,\tau_{\zeta,ns})$ are formally smooth over $\Oc$ of relative dimension 4
    while \[R^\square(\rhobar,\tau_{\zeta_1,\zeta_2}) \cong \Oc[[X_1,\ldots,X_5]]/(X_1^2X_2 -
    (\zeta_1-\zeta_2)^2)\] is a complete intersection domain of relative dimension 4 over
    $\Oc$.

    If $\af_{nr}$ and $\af_N$ are the prime ideals of $\Rbarbox(\rhobar)$ corresponding to
    $\Rbarbox(\rhobar,\tau_{1,s})$ and $\Rbarbox(\rhobar,\tau_{1,ns})$
    respectively, then
    \begin{equation}
      \label{eq:q-1case2}
      Z^4(\Rbarbox(\rhobar,\tau)) =
      \begin{cases}
        [\af_{nr}] & \text{if $\tau = \tau_{\zeta,s}$} \\
        [\af_N] & \text{if $\tau = \tau_{\zeta,ns}$}\\
        2[\af_{nr}] + [\af_N] & \text{if $\tau = \tau_{\zeta_1,\zeta_2}$.}
      \end{cases}
    \end{equation}

  \item If $\bar{x} =\bar{y} = 0$, then $R^\square(\rhobar,\tau_{\zeta,s})$ is formally smooth over
    $\Oc$ of relative dimension 4, $R^\square(\rhobar,\tau_{\zeta,ns})$ is a non-Gorenstein
    Cohen--Macaulay domain of relative dimension 4 over $\Oc$, while
    $R^\square(\rhobar,\tau_{\zeta_1,\zeta_2})$ is a non-Gorenstein Cohen--Macaulay
    domain of relative dimension 4 over $\Oc$.

    Both $\Rbarbox(\rhobar,\tau_{\zeta,s})$ and
    $\Rbarbox(\rhobar,\tau_{\zeta,ns})$ are domains; let the corresponding primes of
    $\Rbarbox(\rhobar)$ be $\af_{nr}$ and $\af_N$ respectively. Then
     \begin{equation}
      \label{eq:q-1case3}
      Z^4(\Rbarbox(\rhobar,\tau)) =
      \begin{cases}
        [\af_{nr}] & \text{if $\tau = \tau_{\zeta,s}$} \\
        [\af_N] & \text{if $\tau = \tau_{\zeta,ns}$}\\
        2[\af_{nr}] + [\af_N] & \text{if $\tau = \tau_{\zeta_1,\zeta_2}$.}
      \end{cases}
    \end{equation} 
  \end{enumerate}
\end{proposition}
\begin{proof}
  Write 
  \begin{align*}
    \rho^\square(\sigma) &= \twomat{1+A}{x+B}{C}{1+D} \\
    \rho^\square(\phi) &= \twomat{1+P}{y+R}{S}{1+Q}
  \end{align*}
  with $A,B,C,D,P,Q,R,S \in \mf$ and $x,y$ lifts of $\bar{x},\bar{y}$ (taken to be zero if $\bar{x}$
  or $\bar{y} = 0$).
  
  First, we have that $R^\square(\rhobar,\tau_{\zeta,s}) = 0$ if $\bar{x} \neq 0$ and 
  \[R^\square(\rhobar,\tau_{\zeta,s}) \cong \Oc[[P,Q,R,S]]\] otherwise.

  Next, we look at $R^\square(\rhobar,\tau_{\zeta_1,\zeta_2})$ for $\zeta_1$ and $\zeta_2$ distinct $l^a$th
  roots of unity.  The condition that $\rho^\square(\sigma)$ has characteristic polynomial
  $(t-\zeta_1)(t-\zeta_2)$ is equivalent to the equations
  \begin{align*}
    A+D &= \zeta_1 + \zeta_2 - 2 \\
    \intertext{and}
    AD - (x+B)C &= (\zeta_1 - 1)(\zeta_2 - 1).
  \end{align*}
  Since $(t-\zeta_1)(t-\zeta_2) \mid t^{q-1} - 1$, by the Cayley--Hamilton theorem we have
  \[\rho^\square(\sigma)^q=\rho^\square(\sigma)\] on $R^\square(\rhobar,\tau_{\zeta_1,\zeta_2})^{\circ}$.
  So the relation $\phi\sigma\phi^{-1} = \sigma^{q}$ yields:
  \[ \twomat{1+A}{x+B}{C}{1+D}\twomat{1+P}{y+R}{S}{1+Q} = \twomat{1+P}{y+R}{S}{1+Q}
  \twomat{1+A}{x+B}{C}{1+D}.\]
  Equating coefficients, eliminating $D$ and writing $U = P-Q$ and \[F = A-D = 2A - (\zeta_1 + \zeta_2 - 2)\] we see that 
  $R^\square(\rhobar,\tau_{\zeta_1,\zeta_2})$ is the reduced, $l$-torsion--free quotient of
  $\Oc[[B,C,F,P,R,S,U]]$ by the relations:
  \begin{align}
    \label{eq:2}
    (x+B)S &= (y+R)C \\
    F(y+R) &= U(x+B) \\
    FS     &= UC \\
    (\zeta_1 - \zeta_2)^2 &= F^2 + 4(x+B)C. \label{eq:3}
  \end{align}

  If $\bar{x} \neq 0$ then these equations are equivalent to $U = F(y+R)(x+B)^{-1}$, $C =
  \tfrac{1}{4}\left((\zeta_1 - \zeta_2)^2 - F^2\right)(x+B)^{-1}$ and $S = C(y+R)(x+B)^{-1}$, so
  that
  \[R^\square(\rhobar,\tau_{\zeta_1,\zeta_2}) \cong \Oc[[B,F,P,R]].\]

  If $\bar{x}  = 0$ and $\bar{y} \neq 0$, then $F = BU(y+R)^{-1}$ and $C = BS(y+R)^{-1}$ will be a
  solution to the equations \eqref{eq:2} to \eqref{eq:3} provided that 
\[ (\zeta_1 - \zeta_2)^2  = \left(\frac{B}{y+R}\right)^2(U^2 + 4(y+R)S);\]
  writing $(X_1,\ldots,X_5) = (B(y+R)^{-1},U^2 + 4(y+R)S,P,R,U)$ we get 
\[R^\square(\rhobar,\tau_{\zeta_1,\zeta_2}) \cong \frac{\Oc[[X_1,\ldots, X_5]]}{X_1^2X_2 -
  (\zeta_1-\zeta_2)^2}\] as claimed.  The other statements about
$R^\square(\rhobar,\tau_{\zeta_1,\zeta_2})$ follow easily.
  
If $\bar{x} = \bar{y} = 0$, then let $\Ac = \Oc[[B,C,F,P,R,S,U]]$ and $I \normal \Ac$ be the ideal:
\[I = \left((\zeta_1 - \zeta_2)^2 - F^2 - 4BC, BS-CR,FR-BU,FS-CU\right).\]
Note that the ideal
\[J = (BS-CR,FR-BU,FS-CU)\] is generated by the $2\times2$ minors of $\begin{pmatrix} B & C & F \\ R
  & S & U \end{pmatrix}$.  So, by proposition \ref{prop:minors}, $\Ac/J$ is a Cohen--Macaulay,
non-Gorenstein domain.  Since $F^2 - 4BC$ is not zero in the domain $\Ac/J \otimes \FF$, $(\lambda,
F^2 - 4BC)$ is a regular sequence in $\Ac/J$.  Hence $(F^2 - 4BC - (\zeta_1 - \zeta_2)^2,\lambda)$
is a regular sequence in $\Ac/J$, and therefore $\Ac/I$ is $\Oc$-flat, Cohen--Macaulay and
non--Gorenstein.  It is reduced because it is Cohen--Macaulay and, as we shall show in the next
paragraph, generically reduced.

To show that $\Ac/I$ is irreducible, it suffices to show that $\Xc = \Spec (\Ac/I \otimes E)$ is
irreducible.  This follows if we can show that $\Xc$ is formally smooth and connected.  As $F^2 -
4BC \neq 0$ on $\Xc$, it is covered by the affine open subsets $\Uc_B = \{B \neq 0\}$ and $\Uc_F =
\{F \neq 0\}$. By the argument used in the $\bar{x} \neq 0$
case, $\Uc_B$ is formally smooth.  A similar argument works for $\Uc_F$: the projection map 
\[p : \Xc \rarrow \Spec\left(\frac{\Oc[[F,B,C,U,P]]}{(F^2 + 4BC - (\zeta_1-\zeta_2)^2)} \otimes
  E\right)\] is an isomorphism from $\Uc_F$ onto an open subscheme; but the right hand side is
easily seen to be formally smooth.  Hence $\Xc$ is formally smooth.  Note that the composition of
the map $p$ with the projection away from $U$ is a
continuous map with connected fibres and connected image, which admits a continuous section
(obtained by taking $R = S = U = 0$); it follows that $\Xc$ is connected, as required.  Since $\Xc$
is formally smooth it is certainly reduced; therefore $\Ac/I$ is generically reduced (as it is
$\Oc$-flat), just as we claimed above.

Now we turn to $R^\square(\rhobar,\tau_{\zeta,ns})$.  By lemma \ref{twisting} we may assume that $\zeta = 1$.
The condition that the characteristic polynomial of $\rho^\square(\sigma)$ be $(t-1)^2$ is
equivalent to the equations:
\begin{align*}
  A+D &= 0 \\
  AD - (x+B)C &= 0.
\end{align*}
Writing $T = P+Q$ and $U = P-Q$, the condition that \[q \tr(\rho^\square(\phi))^2 = (q+1)^2
\det(\rho^\square(\phi))\] becomes
\[(q-1)^2(T+2)^2 = (q+1)^2(U^2  +4(y+R)S).\]
Since $t^q - 1 \equiv q(t-1)\mod (t-1)^2$, the Cayley--Hamilton theorem shows that \[\rho^\square(\sigma)^q
- 1 = q(\rho^\square(\sigma)-1)\] on $R^\square(\rhobar,\tau_{1,ns})$.  From $\phi\sigma\phi^{-1} =
\sigma^q$ we therefore get the equation
\[ (\phi-1)(\sigma - 1) - (\sigma-1)(\phi-1) = (q-1)(\sigma-1)\phi\]
on $R^\square(\rhobar,\tau_{1,ns})$. Equating coefficients and substituting $D=-A$ we
get the equations
\begin{align}
  A^2 + (x+B)C &= 0 \label{eq:3a} \\
  (q-1)^2(T+2)^2 &= (q+1)^2(U^2 + 4(y+R)S) \label{eq:4} \\
  C(y+R) - S(x+B) &= (q-1)(A(1+P) + (x+B)S) \label{eq:5} \\
  U(x+B) - 2A(y+R) &= (q-1)(A(y+R) + (x+B)(1+Q)) \label{eq:6} \\
  2AS - CU &= (q-1)(C(1+P) - AS) \label{eq:7} \\
  S(x+B) - C(y+R) &= (q-1)(C(y+R) - A(1+Q)). \label{eq:8} 
\end{align}
After replacing $P$ with $\frac{T+U}{2}$ and $Q$ with $\frac{T-U}{2}$, this is a complete set of
equations for $R^\square(\rhobar,\tau_{1,ns})$ in $\Oc[[A,B,C,R,S,T,U]]$.

We replace equations \eqref{eq:5} and \eqref{eq:8} by their sum and difference:
\begin{align}
  (q-1)(AU + (x+B)S + C(y+R)) &= 0 \label{eq:9}\\
  (q+1)(C(y+R) - (x+B)S) &= (q-1)A(2+T) \label{eq:10}.
\end{align}

As $R^\square(\rhobar,\tau_{1,ns})$ is $\lambda$-torsion free, equation \eqref{eq:9} implies that 
\begin{equation}
  \label{eq:11}
  AU + (x+B)S + C(y+R) = 0.
\end{equation}
We could also write this equation as $\tr((\sigma-1)\phi) = 0$.

Putting $\alpha(T) = \frac{(q-1)(2+T)}{q+1}$, we find that equations~\eqref{eq:3a}, \eqref{eq:4},
\eqref{eq:6},\eqref{eq:7} and [\eqref{eq:10} and \eqref{eq:11}] may respectively be rewritten:
\begin{align*}
    A^2 + (x+B)C  &= 0 \\
    4(y+R)S + (U-\alpha(T))(U + \alpha(T)) &= 0 \\
    2A(y+R) -(x+B)(U - \alpha(T))&=0 \\
    2AS - C(U + \alpha(T)) &= 0\\
    2C(y+R)  + A(U - \alpha(T)) &= 0 \\
    2(x+B)S + A(U+\alpha(T))&= 0.
      \end{align*}
Let $I$ be the ideal of $\Oc[[A,B,C,R,S,T,U]]$ generated by these equations and let $R' =
\Oc[[A,B,C,R,S,T,U]]/I$, so that $R^\square(\rhobar,\tau_{1,ns})$ is the maximal reduced $l$-torsion
free quotient of $R'$.

If $\bar{x} \neq 0$ then $C$, $U$ and $S$ are uniquely determined by $A$, $B$, $R$ and $T$ so that 
\[R^\square(\rhobar,\tau_{1,ns}) \cong \Oc[[A,B,R,T]].\]
If $\bar{y} \neq 0$, then $S$, $C$ and $A$ are uniquely determined by $B$, $R$, $T$ and $U$ so that 
\[R^\square(\rhobar,\tau_{1,ns}) \cong \Oc[[B,R,T,U]].\]

If $\bar{x} = \bar{y} = 0$, so that $x=  y = 0$, observe that 
\[R' \cong \frac{\Bc}{J_0 + J_1} \] where \[\Bc = \Oc[[X_1,\ldots,X_4,Y_1,\ldots,Y_4,T]],\] the
ideal $J_0$ is generated by the $2\times 2$ minors of \[\begin{pmatrix} X_1 & X_2 & X_3 & X_4 \\ Y_1
  & Y_2 & Y_3 & Y_4 \end{pmatrix}\] and $J_1 = (X_1 + Y_2,Y_3 - X_4 + 2\frac{q-1}{q+1})$.\footnote{There is a typo here
in the published version.}  (The
change of variables is $X_1 = A$, $X_2 = B$, $Y_1 = C$, $Y_2 = -A$, $X_3 = -2R/(2+T)$, $Y_4 =
2S(2+T)$, $Y_3 = (U - \alpha(T))/(2+T)$, and $X_4 = (U + \alpha(T))/(2+T)$.)  Then by proposition
\ref{prop:minors}, $\Bc/J_0$ is a Cohen--Macaulay, non-Gorenstein domain.  Moreover, $(\lambda, X_1
+ Y_2, X_3 - Y_4)$ may be checked to be a regular sequence on $\Bc/J_0$.  Therefore $(X_1 + Y_2, X_3
+ Y_4 + 2\frac{q-1}{q+1},\lambda)$ is also regular, and so $\Bc/(J_0+ J_1)$ is Cohen--Macaulay,
$\Oc$-flat and not Gorenstein.  The same is then true for $R'$.

We show that $R' \otimes \FF$ is a domain, which implies that $R'$ is a domain.  Let $\bar{I}$ be the image of $I$ in
$\FF[[A,B,C,R,S,T,U]]$.  Then $\bar{I}$ is homogeneous so $\gr(R' \otimes \FF) = \FF[A,B,C,R,S,T,U]/\bar{I}$ and it
suffices to check that this is a domain (by \cite{Eisenbud1995-CommutativeAlgebra} corollary 5.5).  It is therefore
sufficient to check that $\Proj(\gr(R'\otimes \FF))$ is reduced and irreducible.\footnote{This argument is not quite
  correct, see Section~\ref{sec:erratum} for a correction. The error originates with me, not
  \cite{Taylor2009-ModularityCourse}.}  But it is easy to check this on the usual seven affine pieces.  This argument is
from \cite{Taylor2009-ModularityCourse}.

Next we show that $R^\square(\rhobar,\tau_{1,ns})$ is reduced.  In fact, we show that
\[\Yc = \Spec (R^\square(\rhobar,\tau_{1,ns})\tensor E)\] is formally smooth, which implies that
$R^\square(\rhobar,\tau_{1,ns})$ is reduced because it is Cohen--Macaulay and $\Oc$-flat.  For
$\star = B$, $C$, $R$, $S$, $U-\alpha(T)$ or $U + \alpha(T)$ let $\Uc_{\star} = \{ \star \neq 0\}
\subset \Yc$ be the corresponding affine open subscheme.  Then the $\Uc_{\star}$ are an affine open
cover of $\Yc$.  For $\star = B$, $C$, $R$ or $S$ we see that $\Uc_{\star}$ is formally smooth by
the same argument as for the cases $\bar{x} \neq 0$ and $\bar{y} \neq 0$ above.  For $\Uc_{U\pm
  \alpha(T)}$, the projection morphism
\[ p : \Uc_{U - \alpha(T)} \rarrow \Spec\left( \frac{\Oc[[C,R,S,T]]}{4RS - (U + \alpha(T))(U - \alpha(T))}
\tensor E\right)\]
is an isomorphism onto an open subscheme.  But the right hand scheme is easily seen to be formally
smooth as required.

Finally we calculate the $Z^4(\bar{R}(\rhobar,\tau))$.  We do this when $\bar{x} = \bar{y} =
0$, as the other cases are similar but easier.  We have written each
$\Rbarbox(\rhobar,\tau)$ as the quotient of $\FF[[A,B,C,R,S,T,U]]$ by an ideal which we
call $I(\tau)$.  Let us recall the presentations:
\begin{align*}
I(\tau_{\zeta,s}) &= (A,B,C) \\
I(\tau_{\zeta,ns}) &= (A^2 + BC, 4RS + U^2, 2CR + AU, 2BS + AU, 2AR - BU, 2AS - CU) \\
I(\tau_{\zeta_1,\zeta_2}) &= (A^2 + BC, BS - CR, 2AR - BU, 2AS - CU)
\end{align*}
(using that $A + D = 0$ in $\Rbarbox(\rhobar,\tau)$ for each $\tau$, we have eliminated $D$ and
written $F = A-D = 2A$).  We have already shown that $I(\tau_{\zeta,s})$ and $I(\tau_{\zeta,ns})$
are prime --- they are the ideals denoted $\af_{nr}$ and $\af_N$ in the statement of the
theorem.  It is clear that
\[Z^4(\Rbarbox(\rhobar,\tau_{\zeta,s})) = [\af_{nr}]\]
and \[Z^4(\Rbarbox(\rhobar,\tau_{\zeta,ns})) = [\af_{N}].\] Suppose that $\pf$ is a prime
ideal of $\FF[[A,B,C,R,S,T,U]]$ containing $I(\tau_{\zeta_1,\zeta_2})$.  We show that $\pf$
contains $\af_{nr}$ or $\af_N$.  If $B,C\in \pf$ then $A \in \pf$ as $A^2 + BC \in
I(\tau_{\zeta_1,\zeta_2})$, and we have $\af_{nr} \subset \pf$.  Otherwise,
suppose that $B \not \in \pf$.  As $A^2 + BC \in \pf$, either both $A$ and $C$ are in $\pf$ or
neither is.  If $A,C \in \pf$ then from $2AR - BU \in \pf$ we deduce that $U \in \pf$, while from
$BS-  CR \in \pf$ we deduce that $S \in \pf$.  It is then easy to see that $\af_N \subset \pf$.  If
$A,B,C \not \in \pf$ then because $B(2CR + AU)$ and $C(2BS + AU)$ are in $I(\tau_{\zeta_1,\zeta_2})$
we see that $2CR + AU, 2BS + AU \in \pf$.  This implies that $A(4RS + U^2) \in \pf$, and so $4RS + U^2
\in \pf$ and hence $\af_N \subset \pf$ as required.  

To finish, it is easy to check that \[e(\Rbarbox(\rhobar,\tau_{\zeta_1,\zeta_2}),\af_{nr}) = 2\] and
that \[e(\Rbarbox(\rhobar,\tau_{\zeta_1,\zeta_2}),\af_N) = 1,\] and so we get equation
\ref{eq:q-1case3}.
\end{proof}

\subsection{Cohen--Macaulayness}\label{sec:CM} If $\tau_0$ is a
semisimple representation of $I_F$ over $E$, let $R(\rhobar,\tau_0)'$ be the maximal reduced and
$l$-torsion--free quotient of $R(\rhobar)$ all of whose $\bar{E}$-points give rise to
representations $\rho$ of $G_F$ with $\rho|_{I_F}^{ss} \cong \tau_0$.  Then I claim that
$R(\rhobar,\tau_0)'$ is always Cohen--Macaulay.  Indeed, if $\tau_0$ is non-scalar then we have
proved this above. If $\tau_0$ is scalar, then we may twist and assume that it is trivial. If $q
\not\equiv \pm 1 \mod l$, this follows from proposition 5.5.  If $q \equiv 1 \mod l$ then we can
deduce the claim from proposition 5.8 together with exercise 18.13
of~\cite{Eisenbud1995-CommutativeAlgebra}, which says that if $R/I$ and $R/J$ are $d$-dimensional
Cohen--Macaulay quotients of a noetherian local
ring $R$, and $\dim R/(I+J) = d-1$, then $R/(I\cap J)$ is Cohen--Macaulay if and only if $R/(I+J)$
is.  We take $R = R^\square(\rhobar)$, and $I$ and $J$ to be the ideals cutting out
$R^\square(\rhobar,\tau_s)$ and $R^\square(\rhobar,\tau_{ns})$ respectively.  Then $R/I$ and $R/J$
are Cohen--Macaulay, and $R/(I+J)$ is a quotient of the formally smooth ring $R/I$ by the single
equation $q\tr(\rho^\square(\phi))^2 = (q+1)^2\det(\rho^\square(\phi))$, and so is Cohen--Macaulay.
Therefore $R/(I \cap J)$ is Cohen--Macaulay as required.  When $q \equiv -1$ mod $l$ the claim
follows from proposition 5.6 unless $\rhobar$ is the direct sum of the trivial and
cyclotomic characters, in which case we use remark~\ref{rem:CM-q+1}.

For $n$-dimensional representations the unrestricted framed deformation ring $R^\square(\rhobar)$ is
always Cohen--Macaulay (in fact, a complete intersection; this is due to David Helm, building on work of Choi
\cite{Choi2009-LocalDeformationRings}).  It is natural to wonder whether the rings obtained by
fixing the semisimplified restriction to inertia are always Cohen--Macaulay.  Note that they are not
always Gorenstein.

For a discussion of how the Cohen--Macaulay property of local deformation rings can be used to show
that certain global Galois deformation rings are flat over $\Oc$, see section 5 of \cite{1111.3654}.

\section{Reduction of types -- proofs.}
\label{sec:repthy}
The aim of this section is to analyse the reduction modulo $l$ of the $K$-types $\sigma(\tau)$
defined in section \ref{sec:types}, and in particular to prove lemma \ref{lem:K-type-reduction}.  
\subsection{The essentially tame case.}
\label{sec:reduction_tame_types}
Suppose that $\tau = (r_\tau,N_\tau)$ where $r_\tau$ is a tamely ramified, semisimple representation
of $I_F$.  Then $\sigma(\tau)$ is inflated from a representation of $GL_2(k_F)$.  We will always use
the same notation for a representation of $GL_2(k_F)$ and its inflaton to $GL_2(\Oc_F)$.  For this subsection
let $G =GL_2(k_F)$, let $B$ be the subgroup of upper-triangular matrices, let $U$ be the subgroup of
unipotent elements of $B$, let $Z$ be the center of $G$ and fix an embedding $\alpha : k_L^\times
\into G$.  Fix a non-trivial additive character $\psi$ of $U$.  Then we have (see
e.g. \cite{BushnellHenniart2006_GL2LocalLanglands} chapter 6):
\begin{itemize}
\item If $r_\tau = (\mathrm{rec}(\tilde{\chi}) \oplus \mathrm{rec}(\tilde{\chi}))|_{I_F}$ and
  $N_\tau \neq 0$, where
  $\tilde{\chi}|_{\Oc_F^\times}$ is inflated from a character $\chi$ of $k_F^\times$, then
  \[\sigma(\tau) = (\chi\circ\det) \otimes \mathrm{St},\] where $\mathrm{St}$ is the Steinberg
  representation of $G$;
\item If $r_\tau = (\mathrm{rec}(\tilde{\chi}) \oplus \mathrm{rec}(\tilde{\chi}))|_{I_F}$ and
  $N_\tau = 0$, where $\tilde{\chi}|_{\Oc_F^\times}$ is inflated from a character $\chi$ of
  $k_F^\times$, then $\sigma(\tau) = \chi\circ\det$;
\item If $r_\tau = (\mathrm{rec}(\tilde{\chi}_1) \oplus \mathrm{rec}(\tilde{\chi}_2))|_{I_F}$, where
  $\tilde{\chi}_1|_{\Oc_F^\times}$ and $\tilde{\chi}_2|_{\Oc_F^\times}$ are inflated from
  \emph{distinct} characters $\chi_1$ and $\chi_2$ of $k_F^\times$, then
  \[\sigma(\tau)=\mu(\chi_1,\chi_2)\] where $\mu(\chi_1,\chi_2)=\Ind_B^G(\chi_1\otimes \chi_2)$;
\item If $r_\tau = (\Ind_{G_L}^{G_F}\mathrm{rec}(\tilde{\theta}))|_{I_F}$ where
  $\tilde{\theta}|_{\Oc_L^\times}$ is inflated from a character $\theta$ of $k_L^\times$ which is
  not equal to its $\Gal(k_L/k_F)$ conjugate $\theta^c$, then \[\sigma(\tau) = \pi_\theta\] where
  $\pi_\theta = \Ind_{ZU}^G(\theta|_Z\psi) - \Ind_{\alpha(k_L^\times)}^G\theta$ (this virtual
  representation is a genuine irreducible representation that is independent of the choice of $\psi$).  
\end{itemize}

The only isomorphisms between these representations are of the form $\mu(\chi_1,\chi_2) \cong
\mu(\chi_2,\chi_1)$ and $\pi_\theta \cong \pi_{\theta^c}$.  

We want to understand the reductions of these representations modulo $l$, and for this see
\cite{Helm2010-LadicFamiliesGL2}.  We will use analagous notation for representations of $G$ in
characteristic zero and in characteristic $l$; hopefully this will not cause confusion. 

If $q\neq \pm1 \bmod l$, then reduction modulo $l$ is a bijection between irreducible
$\bar{\FF}_l$-representations of $G$ and irreducible $\bar{E}$-representations of $G$, as $G$ has
order $q(q+1)(q-1)^2$ which is coprime to $l$.

If $q = 1 \bmod l$, then the distinct irreducible representations of $GL_2(k_F)$ over $\bar{\FF}$
are $\chi \circ \det$ and $\St \tensor (\chi\circ\det)$ for $\chi : k_F^\times \rarrow
\bar{\FF}^\times$, $\mu(\chi_1,\chi_2)$ for $\chi_1, \chi_2 : k_F^\times \rarrow \bar{\FF}^\times$ a
pair of distinct characters, and $\pi_{\theta}$ for $\theta : k_L^\times \rarrow \bar{\FF}^\times$
character which is not isomorphic to its conjugate.  The notation is all entirely analagous to the
characteristic zero case.  Once again, the only isomorphisms are $\mu(\chi_1,\chi_2) \cong
\mu(\chi_2,\chi_1)$ and $\pi_\theta \cong \pi_{\theta^c}$.  The reductions of the characteristic
zero representations are:
\begin{itemize}
\item $\bar{\chi\circ\det} = \bar{\chi}\circ \det$;
\item $\bar{\St \otimes \chi\circ\det} = \St \otimes (\bar{\chi} \circ \det)$;
\item $\bar{\mu(\chi_1,\chi_2)} = \mu(\bar{\chi}_1,\bar{\chi}_2)$ if $\bar{\chi}_1 \neq \bar{\chi}_2$;
\item $\bar{\mu(\chi_1,\chi_2)} = (\bar{\chi}\circ \det) \oplus \St \otimes (\bar{\chi}\circ \det)$
  if $\bar{\chi_1}  = \bar{\chi_2} = \bar{\chi}$;
\item $\bar{\pi}_{\theta} = \pi_{\bar{\theta}}$. 
\end{itemize}
For the last of these, we must observe that $\theta/\theta^c$ is a character of
$k_L^\times/k_F^\times$, a group which has order $q+1$ and so coprime to $l$ (as $l>2$).  Therefore
if $\theta \neq \theta^c$ then $\bar{\theta} \neq \bar{\theta}^c$.

If $q\equiv -1 \bmod l$, then the distinct irreducible representations are: $\chi \circ \det$ for
$\chi : k_F^\times \rarrow \bar{\FF}^\times$, $\mu(\chi_1,\chi_2)$ for $\chi_1, \chi_2 : k_F^\times
\rarrow \bar{\FF}^\times$ unordered pair of distinct characters, $\pi_\theta$ for $\theta :
k_L^\times \rarrow \bar{\FF}^\times$ a character which is not isomorphic to its conjugate, and
$(\chi\circ \det) \otimes \pi_1$ for $\chi : k_F^\times \rarrow \bar{\FF}^\times$ a character.  This
last needs some explanation: $\pi_1$ is the reduction modulo $l$ of $\pi_\theta$ for any character
$\theta : k_L^\times / k_F^\times \rarrow \bar{E}^\times$ which is not equal to $\theta^c$ but whose reduction modulo $l$ is
trivial.  Once again, the only isomorphisms are $\mu(\chi_1,\chi_2) \cong
\mu(\chi_2,\chi_1)$ and $\pi_\theta \cong \pi_{\theta^c}$.  The reductions of the characteristic 0 representations are:
\begin{itemize}
\item $\bar{\chi\circ\det} = \bar{\chi}\circ \det$;
\item $\bar{\mu(\chi_1,\chi_2)} = \mu(\bar{\chi}_1,\bar{\chi}_2)$;
\item $\bar{\pi}_\theta = \pi_{\bar{\theta}}$ if $\bar{\theta} \neq \bar{\theta}^c$;
\item $\bar{\pi}_\theta = \pi_1 \otimes (\bar{\theta}|_{k_F^\times} \circ \det)$ if $\bar{\theta} = \bar{\theta}^c$;
\item $\bar{\St \otimes (\chi \circ \det)}$ has $\pi_1 \otimes (\bar{\chi}\circ \det)$ as a submodule with
  quotient $\bar{\chi} \circ \det$.
\end{itemize}

In particular, comparing this analysis with lemma \ref{lem:type-lifts} shows that:
\begin{lemma}\label{lem:reduction-tame-types}
  If $\tau = (r,0)$ and $\tau'=(r',0)$ are scalar on $P_F$ but not on $\tilde{P}_F$, then $\bar{\sigma(\tau)}$ and
  $\bar{\sigma(\tau')}$ are irreducible and are isomorphic if and only if $r \equiv
  r'\bmod l$.
\end{lemma}

\subsection{The wild case.}
If $\tau = (r,0)$ and all twists of $r$ are wildly ramified (we say that $\tau$ is `essentially wildly ramified'),
then the following lemma will allow us to show that $\bar{\sigma(\tau)}$ is irreducible.  If $\rho$
is a $\bar{\ZZ}_l$-representation of a group $H$, we write $\rhobar$ for $\rho \tensor \bar{\FF}_l$.

\begin{lemma} \label{lem:irreducibility-lemma} Suppose that $H \vartriangleleft J \subset K$ are
  profinite groups such that $H$ is open in $K$, $H$ has pro-order coprime to $l$, and $J/H$ is
  an abelian $l$-group.  Suppose that $\lambda$ is a $\bar{\ZZ}_l$-representation of $J$, and write
  $\eta$ for the restriction of $\lambda$ to $H$.  Suppose that $\eta$ (and hence $\lambda$) is
  irreducible.  Suppose that if $g \in K$ intertwines $\eta$, then $g\in J$.  Then
  \begin{enumerate}
  \item The representations of $J$ extending $\eta$ are precisely $\lambda_i = \lambda \tensor
    \nu_i$ as $\nu_i$ run through the characters of $J/H$. There is an isomorphism $\Ind_H^J \eta
    \otimes \bar{E} \cong \bigoplus_{i} \lambda_i$.  The unique $\bar{\FF}_l$-representation
    extending $\bar{\eta}$ is $\bar{\lambda}$, and all of the Jordan--H\"{o}lder factors of
    $\Ind_H^J \bar{\eta}$ are isomorphic to $\bar{\lambda}$.
  \item A $\bar{\FF}_l$-representation $\rho$ of $J$ contains $\bar{\lambda}$ as a subrepresentation
    if and only if it contains $\bar{\lambda}$ as a quotient.
  \item The representations $\Ind_J^K\lambda_i$ and $\Ind_J^K\bar{\lambda}$ are irreducible.
    \end{enumerate}
\end{lemma}
\begin{proof}
  \begin{enumerate}
  \item In characteristic 0 we argue as follows.  First note that the representations $\lambda_i$
    are distinct, otherwise $\lambda|_H$ would have a non-scalar endomorphism, contradicting
    Schur's lemma. By Frobenius reciprocity, the $\lambda_i$ are distinct irreducible constituents
    of $\Ind_H^J\eta$. Since the sum of their dimensions is $\dim\Ind_H^J\eta$, they are the only
    irreducible constituents.  By Frobenius reciprocity, any representation extending $\eta$ must
    occur in $\Ind_H^J\eta$ and so must be one of the $\lambda_i$, as required.  In characteristic
    $l$, first note that $\bar{\lambda}$ is irreducible since the pro-order of $H$ is coprime to
    $l$.  It follows from this and the fact that $\bar{\nu}_i$ is trivial for all $i$ that the
    Jordan--H\"{o}lder factors of $\Ind_H^J\bar{\eta}$ are isomorphic to $\bar{\lambda}$.  Frobenius
    reciprocity then implies that $\bar{\lambda}$ is the unique irreducible representation of $J$
    extending $H$.
  \item It follows from part 1 that $\Hom_J(\bar{\lambda},\rho)\neq 0$ if and only if
    $\Hom_J(\Ind_H^J\bar{\eta},\rho)\neq 0$.  By Frobenius reciprocity, this is equivalent to
    $\Hom_H(\bar{\eta},\rho)\neq 0$.  But by the assumption on the pro-order of $H$, $\bar{\FF}_l
    $-representations of $H$ are semisimple, and so this is equivalent to
    $\Hom_H(\rho,\bar{\eta})\neq 0$, which by the same argument is equivalent to
    $\Hom_J(\rho, \Ind_H^J\bar{\eta})\neq 0$.
  \item First, note that $\dim \Hom_K(\Ind_J^K\bar{\lambda},\Ind_J^K\bar{\lambda}) = 1$, by Mackey's
    decomposition formula and the assumption that elements of $K \setminus J$ do not intertwine
    $\eta$.  Now suppose that $\rho$ is an irreducible subrepresentation of $\Ind_J^K\bar{\lambda}$.
    By Frobenius reciprocity and part 2 we may deduce that $\rho$ is also an irreducible quotient of
    $\Ind_J^K\bar{\lambda}$.  The composition \[\Ind_J^K\bar{\lambda} \onto \rho \into
    \Ind_J^K\bar{\lambda}\] is then a non-zero element of
    $\Hom_K(\Ind_J^K\bar{\lambda},\Ind_J^K\bar{\lambda})$, and is therefore scalar.  But this is
    only possible if $\rho = \Ind_J^K\bar{\lambda}$, as required.  The statement about
    $\Ind_J^K\lambda_i$ follows. \qedhere
    \end{enumerate}
\end{proof}

\begin{proposition} \label{prop:inducing_subgroups} Let $\tau = (r,0)$ be an essentially wildly
  ramified inertial type.  Then there exists a subgroup $J\subset K$, an irreducible representation
  $\lambda$ of $J$, and a subgroup $\tilde{J}\vartriangleleft J$, such that
  $(\tilde{J},J,K,\lambda)$ satisfy the hypotheses on $(H,J,K,\lambda)$ in lemma
  \ref{lem:irreducibility-lemma} and such that $\sigma(\tau) = \Ind_J^K\lambda$.
  
  In particular, $\bar{\sigma(\tau)}$ is irreducible.
\end{proposition}
\begin{proof} Suppose first that $r$ is the restriction to $I_F$ of a reducible representation of
  $G_F$.  Then $\sigma(\tau) = \Ind_{K_0(N)}^K\epsilon \tensor (\chi\circ \det)$ for a character
  $\epsilon$ of $\Oc_F^\times$ of exponent $N\geq 2$ and a character $\chi$ of $\Oc_F^\times$.  Let
  $J=K_0(N)$, and let \[\tilde{J} = \left\lbrace \twomat{a}{b}{c}{d} \in J \; : \; a\mbox{ has
      order coprime to $l$ modulo $\pf_F$}\right\rbrace.\] Then $\tilde{J}$, $J$ and $\epsilon$
  satisfy all the required hypotheses --- the only one to check is that $\epsilon |_{\tilde{J}}$ is
  not intertwined by any element of $K\setminus J$.  We deduce this (in somewhat circular fashion)
  from the irreducibility of $\Ind_J^K(\epsilon)$, since this is shorter than a direct proof.  If
  $g\in K$ intertwines $\epsilon|_{\tilde{J}}$, then $\Hom_{\tilde{J}\cap
    g\tilde{J}g^{-1}}(\epsilon,\epsilon^{g})\neq 0$.  By Mackey's formula,
  \[\dim\Hom_{\tilde{J}}(\epsilon, \Ind_{\tilde{J}}^K\epsilon) = \sum_{g\in \tilde{J}\setminus K /
    \tilde{J}} \dim \Hom_{\tilde{J}\cap g\tilde{J}g^{-1}}(\epsilon,\epsilon^g).\] The left hand side
  is in turn equal to $\dim\Hom_K(\Ind_{\tilde{J}}^K\epsilon, \Ind_{\tilde{J}}^K\epsilon)$.  But
  $\Ind_{\tilde{J}}^K\epsilon = \bigoplus_i\Ind_J^K\epsilon_i$ where $\epsilon_i$ are the characters
  of $J$ extending $\epsilon |_{\tilde{J}}$, and by the appendix to \cite{breuil2002}, these
  $\Ind_J^K\epsilon_i$ are irreducible and distinct.  Therefore the left hand side is equal to
  $(J:\tilde{J})$.  The right hand side has a contribution of 1 from each $g\in J/\tilde{J}$, and
  therefore from no other $g$, as required.

  Now suppose that $r$ is the restriction to $I_F$ of an irreducible representation of $G_F$.  Then
  $\sigma(\tau) = \Ind_{J}^K\lambda$ for an irreducible representation $\lambda$ of $J$ extending an
  irreducible representation $\eta$ of a pro-$p$ normal subgroup $J^1$ of $J$ (see
  \cite{BushnellHenniart2006_GL2LocalLanglands}, sections 15.5, 15.6 and 15.7 --- note that our $J$
  is the maximal compact subgroup of their $J_\alpha$, but our $J^1$ agrees with their $J_\alpha^1$).
  We have $J/J^1 = k^\times$, where $k$ is the residue field of a quadratic extension of $F$, and so
  $J$ has a normal subgroup $\tilde{J}$ of pro-order coprime to $l$ such that $J/\tilde{J}$ is an
  $l$-group.  Then $(\tilde{J},J,K,\lambda)$ satisfy all the required hypotheses --- the
  intertwining statement follows from \cite{BushnellHenniart2006_GL2LocalLanglands}, 15.6
  Proposition 2.
\end{proof}
\begin{proposition} \label{prop:type_congruence_implies_K-type_congruence} 
  Let $\tau = (r,0)$ and $\tau' = (r',0)$ be inertial types that are not scalar on $\tilde{P}_F$.  If
  $r \equiv r' \bmod l$, then $\bar{\sigma(\tau)}$ and $\bar{\sigma(\tau')}$ are isomorphic.
\end{proposition}
\begin{proof}
  If either of $r$ and $r'$ is (after to a twist) tamely ramified, then so is the other and this
  is contained in lemma \ref{lem:reduction-tame-types}.  Otherwise, by lemma \ref{lem:type-lifts},
  we are in one of the following cases:
  \begin{enumerate}
  
  \item $r = (\chi_1 \oplus \chi_2)|_{I_F}$ for characters $\chi_1$ and $\chi_2$ of $G_F$ that
    are distinct on $P_F$, and $r' = (\chi'_1 \oplus \chi'_2)|_{I_F}$ for characters $\chi'_1$
    and $\chi'_2$ of $G_F$ with $\chi_i \equiv \chi'_i$ for $i=1,2$.
  \item $r = (\Ind_{G_L}^{G_F} \xi)_{I_F}$ and $r'=(\Ind_{G_L}^{G_F}\xi')_{I_F}$ for
    wildly ramified characters $\xi$ and $\xi'$ of $G_L$ such that $\xi \equiv \xi'$,
    and such that $\xi|_{\tilde{P}_F}$ does not extend to $G_F$.
  \item $r|_{\tilde{P_F}}$ is irreducible and $r' = r \otimes \chi$ for a character $\chi$
    of $I_F$ that extends to $G_F$ and such that $\chi \equiv 1$ mod $l$.
  \end{enumerate}

  In the first case, we may write $\chi_i = \rec(\epsilon_i)$ and $\chi'_i =
  \rec(\epsilon'_i)$ with $\epsilon_i$ and $\epsilon'_i$ characters of $F^\times$ such
  that $\epsilon_i \equiv \epsilon'_i \bmod l$ and such that $\epsilon = \epsilon_1/\epsilon_2$ has
  exponent $N \geq 1$.  Since $\epsilon'=\epsilon'_1/\epsilon'_2$ also has exponent $N$, we have
  \begin{align*}
    \sigma(\tau) &= \epsilon_2 \tensor \Ind_{K_0(N)}^K \epsilon \\
    &\equiv \epsilon'_2 \tensor \Ind_{K_0(N)}^K \epsilon' \mod l  \\
    & = \sigma(\tau').\\
  \end{align*}

  In the second case, by twisting we may reduce to the case where $(L/F,\rec^{-1}(\xi))$ is an
  unramified minimal admissible pair (\cite{BushnellHenniart2006_GL2LocalLanglands} paragraph 19.6).
  Then, following through the explicit construction of \cite{BushnellHenniart2006_GL2LocalLanglands}
  paragraphs 19.3 and 19.4, we see that there are:
  \begin{enumerate}
  \item a simple stratum $(\Af,n,\alpha)$ with associated compact open subgroups $J_1 \subset J \subset K$,
    with $J_1$ pro-$p$ and $J/J_1 \cong k_L^\times$;
  \item a representation $\eta$ of $J^1$ and extensions $\lambda$ and
    $\lambda'$ of $\eta$ to $J$ such that $\Ind_J^K(\lambda) = \sigma(\tau)$ and $\Ind_J^K(\lambda') = \sigma(\tau')$.
  \end{enumerate}
  Indeed, up to conjugacy $(\Af,n,\alpha)$, $J_1$ and $\eta$ are determined by
  $\rec^{-1}(\xi)|_{U^1_L}=\rec^{-1}(\xi')|_{U^1_L}$.  The representations $\lambda$ and $\lambda'$
  are defined in terms of $\rec^{-1}(\xi)$ and $\rec^{-1}(\xi')$ by the formulae of
  \cite{BushnellHenniart2006_GL2LocalLanglands} 19.3.1 and corollary 19.4 (together with the
  correction factor of paragraph 34.4, an unramified twist $\Delta_\xi$, that makes no difference to
  the argument).  It is clear from these that if $\xi \equiv \xi'$ then $\lambda \equiv \lambda'$ as
  required.
  
  In the final case, $r' = r \otimes \chi$ for a character $\chi$ of $I_F$ that extends to $G_F$.
  By compatibility of $\tau \mapsto
  \sigma(\tau)$ with twisting,
  \begin{align*}
    \sigma(\tau') &= \sigma(\tau) \otimes \rec^{-1}(\chi)\circ \det\\
    & \equiv \sigma(\tau) \mod l
  \end{align*} as required.
  \end{proof}
\begin{proposition} \label{prop:K-type_congruence_implies_type_congruence} 
  Let $\tau = (r,0)$ and $\tau' = (r',0)$ be inertial types that are not scalar on $\tilde{P}_F$.  If
  $\bar{\sigma(\tau)}$ and $\bar{\sigma(\tau')}$ are isomorphic, then $r \equiv r' \mod l$.
\end{proposition} 
\begin{proof}
  If one of $r$ and $r'$ has a twist which is trivial on $P_F$, then so does the other and in this
  case the proposition follows from \ref{lem:reduction-tame-types}.

  Otherwise may, by twisting, assume that $\sigma(\tau)$ and $\sigma(\tau')$ satisfy $l(\sigma) \leq
  l(\sigma \otimes \chi$) for all characters $\chi$ of $\Oc_F^\times$ (the definition of $l(\sigma)$
  is as in \cite{BushnellHenniart2006_GL2LocalLanglands} paragraph 12.6).  In this case
  $\sigma(\tau)$ and $\sigma(\tau')$ contain the same, non-empty, sets of fundamental strata
  (because this only depends on the restriction to pro-$p$ subgroups).

  If one of $\sigma(\tau)$ and $\sigma(\tau')$ contains a split fundamental stratum
  (\cite{BushnellHenniart2006_GL2LocalLanglands} 13.2) then so does the other.  In this case,
  \cite{BushnellHenniart2006_GL2LocalLanglands} corollary 13.3 implies that they cannot be cuspidal
  types and so we must have $\sigma(\tau) = \Ind_{K_0(N)}^K(\epsilon)$ and $\sigma(\tau') =
  \Ind_{K_0(N')}^K(\epsilon')$ for some $\epsilon$ and $\epsilon'$ of exponents $N$ and $N'$.  It is
  easy to see that in fact we must have $N = N'$.  From lemma \ref{lem:irreducibility-lemma} we
  deduce that $\epsilon \equiv \epsilon' \mod l$, and so $\tau \equiv \tau' \mod l$ as required.
  
  Otherwise, $\sigma(\tau) = \Ind_J^K \lambda$ and $\sigma(\tau') = \Ind_J^K \lambda'$ for a simple
  stratum $(\Af,n,\alpha)$ with associated groups $J^1 \subset J$ and representations $\lambda$ and
  $\lambda'$ extending the representation $\eta$ of $J$.  From lemma \ref{lem:irreducibility-lemma}
  we deduce that $\lambda' = \lambda \otimes \eta$ for a character $\eta$ of $J/J^1$ with $\eta
  \equiv 1 \mod l$.

  If $\Af$ is unramified, then by the reverse of the argument in the second case of the previous
  proposition we see that $\tau = (\Ind_{G_L}^{G_F} \xi)|_{I_F}$ and $\tau = (\Ind_{G_L}^{G_F}
  \xi)|_{I_F}$ for $\xi$ and $\xi'$ characters of $G_L$ with $\xi|_{I_L} \equiv \xi'|_{I_L}$, whence
  the result.  
  
  If $\Af$ is ramified, then $\eta$ can be regarded as a character of $J/J^1 \cong k_M^\times =
  k_F^\times$ with $\eta \equiv 1 \mod l$ for some \emph{ramified} quadratic extension $M/F$.  I
  claim that there is a character $\chi$ of $\Oc_F^\times$ with $\eta = \chi \circ \det$ and $\chi
  \equiv 1 \mod l$.  Indeed, as $l > 2$ we can take the inflation to $\Oc_F^\times$ of the character
  $\chi$ of $k_F^\times$ satisfying $\chi \equiv 1 \mod l$ and $\chi^2 = \eta$.  Then $\sigma(\tau)
  = \sigma(\tau') \otimes (\chi \circ \det)$ and so \begin{align*}\tau &= \tau' \otimes \rec(\chi)
    \\
    &\equiv \tau' \mod l\end{align*} as required.
\end{proof}

\section{Erratum}  
\label{sec:erratum}
The proof of Proposition~\ref{prop:minors} is not correct; however, the proposition is true and the results of the paper
are unaffected.  There is a related gap in the proof of Proposition~\ref{q-1}, which we also fill. I am very grateful to
Lue Pan for pointing out the error.

The problem is that $\Proj(S/(X_1 - \alpha_1Y_1, \ldots, X_i - \alpha_i Y_i))$ being reduced doesn't imply that $S/(X_1 -
\alpha_1Y_1, \ldots, X_i - \alpha_i Y_i)$ is reduced --- there may be nilpotent elements annihilated by the `irrelevant
ideal' generated by positively graded elements.

However, the given reference (\cite{Eisenbud1995-CommutativeAlgebra} Theorem~18.18) certainly implies that $S/I$ is
Cohen--Macaulay; it follows that the claimed sequence is in fact a regular sequence, and the characterisation of when
$S/I$ is Gorenstein follows as in the given proof.

Alternatively we can use an argument that I learned from the MathOverflow posts \cite{127053} and \cite{235179}.  It is
well-known that $R$ is the homogeneous coordinate ring of the image $X$ of the Segre embedding of
$s : \PP^1 \times \PP^{n-1} \rarrow \PP^{2n -1}$.  Then $R$ is Cohen--Macaulay if and only if
$H^i(\PP^{2n-1}, \Ic_X(r)) = 0$ for all $0 < i < n$ and all $r \in \ZZ$, and $R$ is Gorenstein if, in addition,
$\omega_X \cong \Oc_{\PP^{2n-1}}(r)|_X$ for some $r \in \ZZ$ (see \cite[pp9-11, proposition
4.1.1]{migliore1998introduction}).  From the exact sequence
\[0 \rarrow \Ic_X \rarrow \Oc_{\PP^{2n-1}} \rarrow \Oc_X,\]
the equation $s^*\Oc_{\PP^{2n-1}}(1) \cong \Oc_{\PP^{1}}(1) \boxtimes \Oc_{\PP^{n-1}}(1)$ and the K\"{u}nneth formula we
see that $R$ is Cohen--Macaulay.  Since $\omega_X \cong \Oc_{\PP^1}(-2) \boxtimes \Oc_{\PP^{n-1}}(-n)$, we see that $R$
is Gorenstein if and only if $n = 2$.

A similar issue affects the proof of Proposition~\ref{q-1}, in the sentence ``It is therefore sufficient to check that
$\Proj(\gr(R'\otimes \FF))$ is reduced and irreducible.''.  It is not.  However, the given argument shows that
$\gr(R' \otimes \FF)$ has a unique minimal prime ideal and that any nilpotent elements are supported at the irrelevant
ideal.  But we know that this ring is Cohen--Macaulay and so has no embedded associated primes.  It follows that
$\gr(R' \otimes \FF)$ is reduced with a unique minimal prime ideal, and is therefore a domain as claimed.

\bibliography{references,mo}{} \bibliographystyle{amsalpha}
\end{document}